\documentclass[a4paper,oneside,10pt]{article}%
\usepackage{amsmath}
\usepackage{amsfonts}
\usepackage{amssymb}
\usepackage{graphicx}
\usepackage{color}
\usepackage[square,numbers,sort&compress]{natbib}%
\providecommand{\U}[1]{\protect\rule{.1in}{.1in}}
\newtheorem{theorem}{Theorem}[section]

\newtheorem{corollary}[theorem]{Corollary}

\newtheorem{definition}[theorem]{Definition}
\newtheorem{assumption}[theorem]{Assumption}

\newtheorem{proposition}[theorem]{Proposition}
\newtheorem{remark}[theorem]{Remark}

\newenvironment{proof}[1][Proof]{\noindent \textbf{#1.} }{\  \rule{0.5em}{0.5em}}
\numberwithin{equation}{section}

\begin{document}

\title{Stochastic maximum principle, dynamic programming principle, and their
relationship for fully coupled forward-backward stochastic controlled systems}
\author{Mingshang Hu \thanks{Zhongtai Securities Institute for Financial Studies,
Shandong University, Jinan, Shandong 250100, PR China. humingshang@sdu.edu.cn.
Research supported by NSF (No. 11671231) and Young Scholars Program of
Shandong University (No. 2016WLJH10). }
\and Shaolin Ji\thanks{Zhongtai Securities Institute for Financial Studies,
Shandong University, Jinan, Shandong 250100, PR China. jsl@sdu.edu.cn
(Corresponding author). Research supported by NSF (No. 11171187, 11222110 and
11221061), Programme of Introducing Talents of Discipline to Universities of
China (No.B12023). Hu and Ji's research was partially supported by NSF (No.
10921101) and by the 111 Project (No. B12023) }
\and Xiaole Xue\thanks{Zhongtai Securities Institute for Financial Studies,
Shandong University, Jinan 250100, China. Email: xiaolexue1989@gmail.com,
xuexiaole.good@163.com.} }
\maketitle

\textbf{Abstract}. Within the framework of viscosity solution, we study the
relationship between the maximum principle (MP) in \cite{Hu-JX} and the
dynamic programming principle (DPP) in \cite{Hu-JX-DPP} for a fully coupled
forward-backward stochastic controlled system (FBSCS) with a nonconvex control
domain. For a fully coupled FBSCS, both the corresponding MP and the corresponding Hamilton-Jacobi-Bellman (HJB) equation combine an
algebra equation respectively. So this relationship becomes more complicated and almost no work
involves this issue. With the help of a new decoupling technique, we obtain
the desirable estimates for the fully coupled forward-backward variational
equations and establish the relationship. Furthermore, for the smooth case, we
discover the connection between the derivatives of the solution to the algebra
equation and some terms in the first and second-order adjoint equations.
Finally, we study the local case under the monotonicity conditions as in
\cite{Wu98, Li-W} and obtain the relationship between the MP in \cite{Wu98}
and the DPP in \cite{Li-W}.



\addcontentsline{toc}{section}{\hspace*{1.8em}Abstract}

\section{Introduction}

It is well-known that Pontryagin's maximum principle (MP) and Bellman's
dynamic programming principle (DPP) are two of the most important approaches
in solving optimal control problems and there exist close relationship between
them. The relation between the MP and the DPP will help us understand the MP
and the DPP in a more profound way and is studied in many literatures (see
\cite{Yong-Zhou}, \cite{Nie-SW-2} and the references therein). The results on
their connection for deterministic optimal control problems can be seen in
Fleming and Rishel \cite{Fleming}, Barron and Jensen \cite{Barron} and Zhou
\cite{Zhou-1990-1}. For stochastic optimal control problems, the classical
results on the relationship between the MP and the DPP were studied by
Bensoussan \cite{Bensoussan}. Within the framework of viscosity solution, Zhou
\cite{Zhou-1990-2, Zhou-1991} obtained the relation between these two approaches.

In this paper, we study the relationship between the MP and the DPP for a
stochastic optimal control problem where the system is governed by the
following controlled fully coupled forward-backward stochastic differential
equation (FBSDE):
\begin{equation}
\left\{
\begin{array}
[c]{rl}%
dX(t)= & b(t,X(t),Y(t),Z(t),u(t))dt+\sigma(t,X(t),Y(t),Z(t),u(t))dB(t),\\
dY(t)= & -g(t,X(t),Y(t),Z(t),u(t))dt+Z(t)dB(t),\;t\in\lbrack0,T],\\
X(0)= & x_{0},\ Y(T)=\phi(X(T)),
\end{array}
\right.  \label{intro--fbsde}%
\end{equation}
and the cost functional is defined by the solution to the backward stochastic
differential equation (BSDE) in (\ref{intro--fbsde}) at time $0$, i.e.,
\begin{equation}
J(u(\cdot))=Y(0). \label{intro-cost}%
\end{equation}
When the coefficients of the forward stochastic differential equation (SDE) in
(\ref{intro--fbsde}) are independent of the terms $Y(\cdot)$ and $Z(\cdot)$,
we call (\ref{intro--fbsde}) with the cost functional (\ref{intro-cost}) a
decoupled forward-backward stochastic controlled system (FBSCS). FBSCSs can be
used to describe some important problems in mathematical finance and
stochastic control theory. For example, the portfolios of a large investor,
the generalized stochastic recursive utilities of consumers, the
leader-follower stochastic differential games and principal-agent problems may
involve in solving optimal controls for fully coupled FBSCSs
\cite{Cvi-Zhang,Elkaroui-PQ,Yong}.

Peng \cite{Peng-1993} first established a local stochastic maximum principle
for the decoupled FBSCS. Then the local stochastic maximum principles for
other various problems were studied in Dokuchaev and Zhou
\cite{Dokuchaev-Zhou}, Ji and Zhou \cite{Ji-Zhou} and Shi and Wu \cite{Shi-Wu}
(see also the references therein). When the control domain is nonconvex, the
global stochastic maximum principle for a decoupled FBSCS has not been
obtained for a long time since Peng \cite{Peng-open} proposed it as an open
problem. For this open problem, Yong \cite{Yong-2010} and Wu \cite{Wu-2013}
derived stochastic maximum principles which contain unknown parameters. Hu
\cite{Hu-2017} studied this decoupled FBSCS and obtained the first-order and second-order
variational equations for the BSDE in (\ref{intro--fbsde}) which leads to a
completely novel global maximum principle. Recently, Hu, Ji and Xue
\cite{Hu-JX} obtained a global stochastic maximum principle for the fully
coupled FBSCS. In contrast with the progresses in deriving stochastic maximum
principles, Peng \cite{Peng-1992, Peng-lecture} deduced the DPP and introduced
the generalized Hamilton-Jacobi-Bellman (HJB) equation for a decoupled FBSCS.
Under monotonicity conditions, Li and Wei \cite{Li-W} and Li \cite{Li jun}
built the DPP and proved that the value function is a viscosity solution to
the generalized HJB equation for a fully coupled FBSCS. Then, by establishing
the DPP and various properties of the value function, Hu, Ji and Xue
\cite{Hu-JX-DPP} studied the existence and uniqueness of viscosity solutions
to the generalized HJB equation for a fully coupled FBSCS.

As for the relationship between the MP and the DPP for the decoupled FBSCS,
Shi \cite{Shi} and Shi and Yu \cite{Shi-Y} investigated the local case in
which the control domain is convex and the value function is smooth; within
the framework of viscosity solution, Nie, Shi and Wu \cite{Nie-SW-1} studied
the local case; Nie, Shi and Wu \cite{Nie-SW-2} studied the general case with
the help of the first-order and second-order adjoint equations which are introduced
in Hu \cite{Hu-2017}. Up to our knowledge, there are few works about the
connection between the MP and the DPP for fully coupled FBSCSs. Especially,
there is no research results in the case that the diffusion coefficient
$\sigma$ of the forward SDE in (\ref{intro--fbsde}) depends on the term $Z$.

Inspired by the above works, in this paper, we investigate the connection
between the MP and the DPP for fully coupled FBSCSs with a nonconvex control
domain. We obtain that the connection between the adjoint process $(p,P)$ in
the maximum principle in \cite{Hu-JX} and the first-order and second-order
sub- (resp. super-) jets of the value function $W$ in \cite{Hu-JX-DPP} in the
$x$-variable is%
\begin{equation}
\left\{
\begin{array}
[c]{l}%
\{p(s)\}\times\lbrack P(s),\infty)\subseteq D_{x}^{2,+}W(s,\bar{X}%
^{t,x;\bar{u}}(s)),\\
D_{x}^{2,-}W(s,\bar{X}^{t,x;\bar{u}}(s))\subseteq\{p(s)\}\times(-\infty
,P(s)],\text{ }\forall s\in\lbrack t,T],\text{ }P-a.s.
\end{array}
\right.  \label{intr-rel}%
\end{equation}
and the connection between the function $\mathcal{H}_{1}$ and the right sub-
(resp. super-) jets of $W$ in the $t$-variable is
\begin{equation}
\left\{
\begin{array}
[c]{l}%
\lbrack\mathcal{H}_{1}(s,\bar{X}^{t,x;\bar{u}}(s),\bar{Y}^{t,x;\bar{u}%
}(s),\bar{Z}^{t,x;\bar{u}}(s)),\infty)\subseteq D_{t+}^{1,+}W(s,X^{t,x;\bar
{u}}(s)),\\
D_{t+}^{1,-}W(s,X^{t,x;\bar{u}}(s))\subseteq(-\infty,\mathcal{H}_{1}(s,\bar
{X}^{t,x;\bar{u}}(s),\bar{Y}^{t,x;\bar{u}}(s),\bar{Z}^{t,x;\bar{u}%
}(s))],\text{ }P-a.s..
\end{array}
\right.  \label{intr-time}%
\end{equation}

Comparing with the results in \cite{Nie-SW-1, Nie-SW-2}, the difficulties of
proving the above relations come from the fully coupled property of our
controlled system. Note that due to the fully coupled property, the MP in \cite{Hu-JX} includes an algebra equation which leads to the adjoint
process $(p,P)$ in \cite{Hu-JX-DPP} becomes much more complex than that in
\cite{Hu-2017},  and the value function $W$ in \cite{Hu-JX-DPP} should satisfy
the HJB equation combined with an algebra equation.  When we establish the
relation (\ref{intr-rel}), we need to perturb the initial state $x$ which
leads to the fully coupled variational equation (\ref{eq-xy-hat}). The key
step in obtaining (\ref{intr-rel}) is to estimate the remainder terms
$\varepsilon_{i}(\cdot)$ ($i=1,2,3$) in (\ref{eq-xy-hat}). But the approach in
\cite{Nie-SW-2} does not work. The reason is that for decoupled case, one can
first estimate the remainder terms of the forward equation, and then estimate
the remainder terms of the backward equation by standard estimates as in
\cite{Nie-SW-2}. But for the fully coupled case, $\hat{Z}$ will appear in the
remainder terms of the forward equation which yields that we can not estimate
the remainder terms in the forward equation firstly.
%Recall that the main
%difficulty to do estimates in our context is that $\mathbb{E}[\int_{t}%
%^{T}|\hat{Z}\left(  s\right)  |^{p}ds]<\infty$ for $p>2$ may not hold (see
%also \cite{Hu-JX,Nie-SW-2}).
To overcome this difficulty, by utilizing the relationship between $(\hat
{Y}(\cdot),\hat{Z}(\cdot))$ and $\hat{X}(\cdot)$ (see (\ref{relation-hat}) and
(\ref{eq-epsilon2})), we propose a new decoupling technique and estimate the
remainder terms of the forward and backward equations simultaneously. Then, we
obtain the desirable estimates which make the establishment of the relation
(\ref{intr-rel}) possible. The idea to prove the relation (\ref{intr-time}) is similar.

When the value function $W$ is supposed to be smooth, we discover two novel
connections:\\
\noindent  (i) the relation between the algebra equation for the MP and the algebra equation for the HJB equation,
 i.e., $\Delta(\cdot)$ in (\ref{def-delt}) and $V(\cdot)$ in (\ref{eq-V}), for $s\in\lbrack t,T]$,
\begin{equation*}
\begin{array}
[c]{rl}
\Delta(s)=&V(s,\bar{X}^{t,x;\bar{u}}(s),W(s,\bar{X}^{t,x;\bar{u}}(s)),W_{x}(s,\bar
{X}^{t,x;\bar{u}}(s)),u)\\
&-V(s,\bar{X}^{t,x;\bar{u}}(s),W(s,\bar{X}^{t,x;\bar{u}}(s)),W_{x}(s,\bar
{X}^{t,x;\bar{u}}(s)),\bar{u}(s)),
\end{array}
\end{equation*}

\noindent  (ii)  the relation between the derivatives of the solution $V$ to the algebra equation
(\ref{eq-V}) and the terms $K_{1}(\cdot)$, $K_{2}(\cdot)$ in the adjoint
equations (\ref{eq-p}) and (\ref{eq-P}), for $s\in\lbrack t,T]$,%
\begin{equation}%
\begin{array}
[c]{rl}%
\left.  \frac{\partial V}{\partial x}(s,x,W(s,x),W_{x}(s,x),\bar{u}(s))\right\vert _{x=\bar
{X}^{t,x;\bar{u}}(s)} & =K_{1}(s),\\
\left.  \frac{\partial^{2}V}{\partial x^{2}}(s,x,W(s,x),W_{x}(s,x),\bar{u}(s))\right\vert
_{x=\bar{X}^{t,x;\bar{u}}(s)} & =\tilde{K}_{2}(s),
\end{array}
\label{connection-algebra}%
\end{equation}
where $K_{1}\left(  \cdot\right)  $ (resp. $\tilde{K}_{2}(\cdot)$) is defined
in (\ref{new-k1}) (resp. (\ref{K2-til})).\\
\noindent From the point of view of the MP,
$K_{1}\left(  \cdot\right)  $ is the coefficient of $X_{1}\left(
\cdot\right)  $ in the first-order variational equation of $Z\left(
\cdot\right)  $ in Lemma 3.13 in \cite{Hu-JX} and it measures the sensitivity
of the variable $Z(\cdot)$ to the variable $X(\cdot)$ under the optimal state.
From the viewpoint of the DPP and the HJB equation,
\[
V(s,\bar{X}^{t,x;\bar{u}}(s),W(s,\bar{X}^{t,x;\bar{u}}(s)),W_{x}(s,\bar
{X}^{t,x;\bar{u}}(s)),\bar{u}(s))=\bar{Z}^{t,x;\bar{u}}(s).
\]
Thus, the connection (\ref{connection-algebra}) is naturally established. In
fact, when $\sigma$ is independent of $y$ and $z$,%
\begin{align*}
&K_{1}\left(  s\right)   =p\left(  s\right)  \sigma_{x}\left(  s\right)
+q\left(  s\right)  ,\;s\in\lbrack t,T];\\
& \left.  \frac{\partial V}{\partial x}(s,x,W(s,x),W_{x}(s,x),\bar{u}(s))\right\vert _{x=\bar
{X}^{t,x;\bar{u}}(s)}\\
&  =W_{x}\left(  s,\bar
{X}^{t,x;\bar{u}}(s)\right)  \sigma_{x}\left(  s\right)  +W_{xx}\left(
s,\bar{X}^{t,x;\bar{u}}(s)\right)  \sigma\left(  s\right)  ,
\end{align*}
which can be directly deduced by the connection between the MP and the DPP for
decoupled FBSCSs. The connection between $\tilde{K}_{2}(\cdot)$ and $\frac{\partial^2 V}{\partial x^2}(\cdot)$
can be analyzed similarly.
%From the viewpoint of the MP, $K_{2}\left(  s\right)  $ in (\ref{new-k2}) is
%the coefficient of the variational equation of $Z$ with respect to $X_{1}%
%^{2}\left(  \cdot\right)  $. And from the viewpoint of the DPP, replace the
%state and control in $V_{xx}\left(  \cdot\right)  $ as the optimal one in the
%smooth case, it is exactly $\tilde{K}_{2}\left(  s\right)  $ not $K_{2}\left(
%s\right)  $, where $\tilde{K}_{2}(s)$ is defined in (\ref{K2-til}). If
%$\sigma$ is independent of $y$ and $z$, then $V_{xx}(s,\bar{X}^{t,x;\bar{u}%
%}(s),\bar{u}(s))=2W_{xx}\left(  s,\bar{X}^{t,x;\bar{u}}(s)\right)  \sigma
%_{x}\left(  s\right)  +W_{xxx}\left(  s,\bar{X}^{t,x;\bar{u}}(s)\right)
%\sigma\left(  s\right)  +W_{x}\left(  s,\bar{X}^{t,x;\bar{u}}(s)\right)
%\sigma_{xx}\left(  s\right)  =\tilde{K}_{2}\left(  s\right)  $ and
%$K_{2}\left(  s\right)  =2P\left(  s\right)  \sigma_{x}\left(  s\right)
%+Q\left(  s\right)  +p\left(  s\right)  \sigma_{xx}\left(  s\right)  $. It is
%clear that, one can obtain the form of $\tilde{K}_{2}\left(  s\right)  $ by
%replacing $W_{xx}\left(  s,\bar{X}^{t,x;\bar{u}}(s)\right)  $ by $P\left(
%s\right)  $, $W_{xxx}\left(  s,\bar{X}^{t,x;\bar{u}}(s)\right)  \sigma\left(
%s\right)  $ by $Q\left(  s\right)  $ and $W_{x}\left(  s,\bar{X}^{t,x;\bar{u}%
%}(s)\right)  $ by $p\left(  s\right)  $.
Besides the smooth case, we also study other special cases. When the diffusion
term $\sigma$ of the forward stochastic differential equation in
(\ref{state-eq}) is linear in $z$, we relax the assumption that $q(\cdot)$ is
bounded. For the so called local case in which the control domain is convex
and compact, the relations in Theorem \ref{th-visco-x} are still hold under
our Assumptions \ref{assum-1}, \ref{assum-2} and \ref{assum-3} since our
control domain is only supposed to be a nonempty and compact set. Then, we
study the local case under the monotonicity conditions as in \cite{Wu98, Li-W}
and obtain the relationship between the MP in \cite{Wu98} and the DPP in
\cite{Li-W} for the fully coupled FBSCS.

The rest of the paper is organized as follows. In section 2, we give the
preliminaries. The connections between the value function and the adjoint
processes within the framework of viscosity solution are given in section 3.
In the last section, we study some special cases.

\section{ Preliminaries}

Let $T>0$ be fixed, and $U\subset\mathbb{R}^{k}$ be nonempty and compact.
Given $t\in\lbrack0,T)$, denote by $\mathcal{U}^{w}[t,T]$ the set of all
5-tuples $(\Omega,\mathcal{F},P,B(\cdot);u(\cdot))$ satisfying the following:

\begin{description}
\item[(i)] $(\Omega,\mathcal{F},P)$ is a complete probability space;

\item[(ii)] $B(r)=(B_{1}(r),B_{2}(r),...B_{d}(r))_{r\geq t}^{\intercal}$ is a
$d$-dimensional standard Brownian motion defined on $(\Omega,\mathcal{F}%
,P)$ with $B(t)=0$ a.s.. Set $\mathbb{F}^{t}:=\left\{  \mathcal{F}_{s}^{t}\right\}
_{s\geq t}$ and $\mathbb{F}=\mathbb{F}^{0}$, where $\mathcal{F}_{s}^{t}$ is the $P$-augmentation of the natural filtration of
$\sigma\{B(r):t\leq r\leq s\}$;

\item[(iii)] $u(\cdot):[t,T]\times\Omega\rightarrow U$ is an $\mathbb{F}^{t}%
$-adapted process on $(\Omega,\mathcal{F},P)$.
\end{description}

When there is no confusion, we also use $u(\cdot)\in\mathcal{U}^{w}[t,T]$.
Denote by $\mathbb{R}^{n}$ the $n$-dimensional real Euclidean space and
$\mathbb{R}^{k\times n}$ the set of $k\times n$ real matrices. Let
$\langle\cdot,\cdot\rangle$ (resp. $\Vert\cdot\Vert$) denote the usual scalar
product (resp. usual norm) of $\mathbb{R}^{n}$ and $\mathbb{R}^{k\times n}$.
The scalar product (resp. norm) of $M=(m_{ij})$, $N=(n_{ij})\in\mathbb{R}%
^{k\times n}$ is denoted by $\langle M,N\rangle=\mathrm{tr}\{MN^{\intercal}\}$
(resp. $\Vert M\Vert=\sqrt{MM^{\intercal}}$), where the superscript
$^{\intercal}$ denotes the transpose of vectors or matrices.

For each given $p\geq1$, we introduce the following spaces.

\noindent  $L^{p}(\mathcal{F}_{T}^{t};\mathbb{R}^{n})$ : the space of $\mathcal{F}%
_{T}^{t}$-measurable $\mathbb{R}^{n}$-valued random vectors $\eta$ such that
\[
||\eta||_{p}:=(\mathbb{E}[|\eta|^{p}])^{\frac{1}{p}}<\infty,
\]

\noindent  $L^{\infty}(\mathcal{F}_{T}^{t};\mathbb{R}^{n})$: the space of $\mathcal{F}%
_{T}^{t}$-measurable $\mathbb{R}^{n}$-valued random vectors $\eta$ such that $$||\eta
||_{\infty}=\mathrm{ess~sup}_{\omega\in\Omega
}|\eta(\omega)|<\infty,$$

\noindent  $L_{\mathbb{F}^{t}}^{p}(t,T;\mathbb{R}^{n})$: the space of $\mathbb{F}^{t}%
$-adapted $\mathbb{R}^{n}$-valued stochastic processes on $[t,T]$ such that
\[
\mathbb{E}[\int_{t}^{T}|f(r)|^{p}dr]<\infty,
\]

\noindent  $L_{\mathbb{F}^{t}}^{\infty}(t,T;\mathbb{R}^{n})$: the space of $\mathbb{F}%
^{t}$-adapted $\mathbb{R}^{n}$-valued stochastic processes on $[t,T]$ such that
\[
||f(\cdot)||_{\infty}=\mathrm{ess~sup}_{(r,\omega)\in\lbrack t,T]\times\Omega
}|f(r,\omega)|<\infty,
\]

\noindent  $L_{\mathbb{F}^{t}}^{p,q}(t,T;\mathbb{R}^{n})$: the space of $\mathbb{F}^{t}%
$-adapted $\mathbb{R}^{n}$-valued stochastic processes on $[t,T]$ such that
\[
||f(\cdot)||_{p,q}=\{\mathbb{E}[(\int_{t}^{T}|f(r)|^{p}dr)^{\frac{q}{p}%
}]\}^{\frac{1}{q}}<\infty,
\]

\noindent $L_{\mathbb{F}^{t}}^{p}(\Omega;C([t,T],\mathbb{R}^{n}))$: the space of
$\mathbb{F}^{t}$-adapted $\mathbb{R}^{n}$-valued continuous stochastic
processes on $[t,T]$ such that
\[
\mathbb{E}[\sup\limits_{t\leq r\leq T}|f(r)|^{p}]<\infty.
\]

To simplify the presentation, we only consider the case $d=1$. The results for
$d>1$ are similar. For each fixed $(t,x)\in\lbrack0,T]\times\mathbb{R}$
and\ $u(\cdot)\in\mathcal{U}^{w}[t,T]$, consider the following controlled
fully coupled FBSDE: for $s\in\lbrack t,T]$,
\begin{equation}
\left\{
\begin{array}
[c]{rl}%
dX^{t,x;u}(s)= & b(s,X^{t,x;u}(s),Y^{t,x;u}(s),Z^{t,x;u}(s),u(s))ds\\
&+\sigma(s,X^{t,x;u}(s),Y^{t,x;u}(s),Z^{t,x;u}(s),u(s))dB(s),\\
dY^{t,x;u}(s)= & -g(s,X^{t,x;u}(s),Y^{t,x;u}(s),Z^{t,x;u}(s),u(s))ds+Z^{t,x;u}%
(s)dB(s),\\
X^{t,x;u}(t)= & x,\ Y^{t,x;u}(T)=\phi(X^{t,x;u}(T)),
\end{array}
\right.  \label{state-eq}%
\end{equation}
where%
\[
b:[t,T]\times\mathbb{R}\times\mathbb{R}\times\mathbb{R}\times U\rightarrow
\mathbb{R},
\]%
\[
\sigma:[t,T]\times\mathbb{R}\times\mathbb{R}\times\mathbb{R}\times
U\rightarrow\mathbb{R},
\]%
\[
g:[t,T]\times\mathbb{R}\times\mathbb{R}\times\mathbb{R}\times U\rightarrow
\mathbb{R},
\]%
\[
\phi:\mathbb{R}\rightarrow\mathbb{R}.
\]

\begin{assumption}
\label{assum-1} (i) $b,\sigma,g,\phi$ are continuous with respect to
$s,x,y,z,u$, and there exist constants $L_{i}>0$, $i=1,2,3$ such that%
\[
|b(s,x_{1},y_{1},z_{1},u)-b(s,x_{2},y_{2},z_{2},u)|\leq L_{1}|x_{1}%
-x_{2}|+L_{2}(|y_{1}-y_{2}|+|z_{1}-z_{2}|),
\]%
\[
|\sigma(s,x_{1},y_{1},z_{1},u)-\sigma(s,x_{2},y_{2},z_{2},u)|\leq L_{1}%
|x_{1}-x_{2}|+L_{2}|y_{1}-y_{2}|+L_{3}|z_{1}-z_{2}|,
\]%
\[
|g(s,x_{1},y_{1},z_{1},u)-g(s,x_{2},y_{2},z_{2},u)|\leq L_{1}(|x_{1}%
-x_{2}|+|y_{1}-y_{2}|+|z_{1}-z_{2}|),
\]%
\[
|\phi(x_{1})-\phi(x_{2})|\leq L_{1}|x_{1}-x_{2}|,
\]
for all $s\in\lbrack0,T],x_{i},y_{i},z_{i}\in\mathbb{R},$ $i=1,2$, $u\in U$.
\end{assumption}

(ii) For any $2\leq\beta\leq8$,\ $\Lambda_{\beta}:=C_{\beta}2^{\beta
+1}(1+T^{\beta})c_{1}^{\beta}<1$, where $c_{1}=\max\{L_{2,}L_{3}\}$,
$C_{\beta}$ is defined in Lemma 5.1 in \cite{Hu-JX}.

\begin{remark}
Since $U$ is compact, from the above assumption (i) we obtain that%
\[
|\psi(s,x,y,z,u)|\leq L(1+|x|+|y|+|z|),
\]
where $L>0$ is a constant and $\psi=b,$ $\sigma,$ $g$ and $\phi$.
\end{remark}

\begin{remark}
Note that $\beta=2$ is sufficient to guarantee the DPP. But, for the MP we
need $2\leq\beta\leq8$.
\end{remark}

Given $u(\cdot)\in\mathcal{U}^{w}[t,T]$, by Theorem 2.2 in \cite{Hu-JX}, the
equation \eqref{state-eq} has a unique solution $(X^{t,x;u}(\cdot
),Y^{t,x;u}(\cdot)$, $Z^{t,x;u}(\cdot))$ $\in L_{\mathbb{F}^{t}}^{\beta
}(\Omega;C([t,T],\mathbb{R}))\times L_{\mathbb{F}^{t}}^{\beta}(\Omega
;C([t,T],\mathbb{R}))\times L_{\mathbb{F}^{t}}^{2,\beta}(t,T;\mathbb{R})$ for $\beta\in [2,8]$.
For the existence and uniqueness of solutions of FBSDEs, the readers may refer
 to (\cite{Hu-Peng95,Ma-WZZ,Ma-Y,Pardoux-Tang}).
For each given $(t,x)\in\lbrack0,$ $T]\times\mathbb{R}$, define the cost
functional
\begin{equation}
J(t,x;u(\cdot))=Y^{t,x;u}(t).
\end{equation}

\begin{remark}
Since the coefficients are deterministic and $u(\cdot)$ is an $\mathbb{F}^{t}%
$-adapted process, $J(t,x;u(\cdot))$ is deterministic.
\end{remark}

For each given $(t,x)\in\lbrack0,T]\times\mathbb{R}$, define the value
function
\begin{equation}%
\begin{array}
[c]{ll}%
W(t,x)= & \underset{u(\cdot)\in\mathcal{U}^{w}[t,T]}{\inf}J(t,x;u(\cdot)).
\end{array}
\label{obje-eq}%
\end{equation}

We introduce the following generalized HJB equation combined with an algebra
equation for $W(\cdot,\cdot)$:
\begin{equation}
\left\{
\begin{array}
[c]{l}%
W_{t}(t,x)+\inf\limits_{u\in U}\{G(t,x,W(t,x),W_{x}(t,x),W_{xx}\left(
t,x\right)  ,u)\}=0,\\
W(T,x)=\phi(x),
\end{array}
\right.  \label{eq-hjb}%
\end{equation}
where
\begin{equation}
\begin{array}
[c]{l}%
G(t,x,v,p,A,u)\\
=pb(t,x,v,V\left(  t,x,v,p,u\right)  ,u)+\frac{1}{2}A(\sigma(t,x,v,V\left(
t,x,v,p,u\right)  ,u))^{2}\\
\ \ \ +g(t,x,v,V\left(  t,x,v,p,u\right)  ,u),\\
V(t,x,v,p,u)=p\sigma(t,x,v,V\left(  t,x,v,p,u\right)  ,u),\\
 \forall (t,x,v,p,A,u)\in\lbrack0,T]\times\mathbb{R}\times\mathbb{R}\times
\mathbb{R}\times\mathbb{R\times}U.
\end{array} \label{eq-V}
\end{equation}

Now, we introduce the following definition of viscosity solution (see
\cite{Crandall-lecture}).

\begin{definition}
(i) A real-valued continuous function $W(\cdot,\cdot)\in C\left(
[0,T]\times\mathbb{R}\right)  $ is called a viscosity subsolution (resp.
supersolution) of (\ref{eq-hjb}) if $W(T,x)\leq\phi(x)$ (resp. $W(T,x)\geq
\phi(x))$ for all $x\in\mathbb{R}$ and if for all $f\in C_{b}^{2,3}\left(
[0,T]\times\mathbb{R}\right)  $ such that $W(t,x)=f(t,x)$ and $W-f$ attains a
local maximum (resp. minimum) at $(t,x)\in\lbrack0,T)\times\mathbb{R}$, we have%
\[
\left\{
\begin{array}
[c]{l}%
f_{t}(t,x)+\inf\limits_{u\in U}\{G(t,x,f(t,x),f_{x}(t,x),f_{xx}\left(
t,x\right)  ,u)\}\geq0\\
(\text{resp. }f_{t}(t,x)+\inf\limits_{u\in U}\{G(t,x,f(t,x),f_{x}%
(t,x),f_{xx}\left(  t,x\right)  ,u)\}\leq0).
\end{array}
\right.
\]

(ii) A real-valued continuous function $W(\cdot,\cdot)\in C\left(
[0,T]\times\mathbb{R}\right)  $ is called a viscosity solution to
(\ref{eq-hjb}), if it is both a viscosity subsolution and viscosity supersolution.
\end{definition}

\begin{remark}
The viscosity solution to (\ref{eq-hjb}) can be equivalently defined by
sub-jets and super-jets (see \cite{Crandall-lecture}).
\end{remark}

\begin{proposition} (see \cite{Hu-JX-DPP})
Let Assumption \ref{assum-1} hold. Then, for each $t\in\lbrack0,T]$ and
$x,x^{\prime}\in\mathbb{R}$,%
\[
|W(t,x)-W(t,x^{\prime})|\leq C|x-x^{\prime}|\text{ and }|W(t,x)|\leq
C(1+|x|),
\]
where $C>0$ depends on $L_{1}$, $L_{2}$, $L_{3}$ and $T$. Moreover, $W(\cdot,\cdot)$ satisfies the DPP. If
$L_{3}$ is small enough, then $W(\cdot,\cdot)$ is the
viscosity solution to \eqref{eq-hjb}.
\end{proposition}

Let $\bar{u}(\cdot)\in\mathcal{U}^{w}[t,T]$ be optimal. Then,
$W(t,x)=J(t,x;\bar{u}(\cdot))$. The corresponding solution $(\bar{X}%
^{t,x;\bar{u}}(\cdot),\bar{Y}^{t,x;\bar{u}}(\cdot)$, $\bar{Z}^{t,x;\bar{u}%
}(\cdot))$ to equation \eqref{state-eq} is called optimal trajectory. To
derive the MP, we give the following assumptions.

\begin{assumption}
\label{assum-2} For $\psi=b,$ $\sigma,$ $g$ and $\phi$, we suppose

(i) $\psi_{x}$, $\psi_{y}$, $\psi_{z}$ are bounded and continuous in
$(x,y,z,u)$; there exists a constant $\ L>0$ such that
\[
|\sigma(t,0,0,z,u)-\sigma(t,0,0,z,u^{\prime})|\leq L(1+|u|+|u^{\prime}|).
\]

(ii) $\psi_{xx}$, $\psi_{xy}$, $\psi_{yy}$ , $\psi_{xz}$, $\psi_{yz}$,
$\psi_{zz}$ are bounded and continuous in $(x,y,z,u)$.
\end{assumption}

\begin{remark}
It is clear that $L_{1}$ in Assumption \ref{assum-1} is $\max\{||b_{x}%
||_{\infty},||\sigma_{x}||_{\infty},||g_{x}||_{\infty},$ $||g_{y}||_{\infty
},||g_{z}||_{\infty},||\phi_{x}||_{\infty}\}$, $L_{2}=\max\{||b_{y}%
||_{\infty},||b_{z}||_{\infty},||\sigma_{y}||_{\infty}\}$ and $L_{3}%
=||\sigma_{z}||_{\infty}$.
\end{remark}

For $\beta_{0}>0$, set
\[
F(y)=L_{1}+\left(  L_{2}+L_{1}+\beta_{0}^{-1}L_{1}L_{2}\right)  |y|+\left[
L_{2}+\beta_{0}^{-1}(L_{1}L_{2}+L_{2}^{2})\right]  y^{2}+\beta_{0}^{-1}%
L_{2}^{2}|y|^{3},\ y\in\mathbb{R}\text{.}%
\]
Let $s(\cdot)$ be the maximal solution to the following ordinary differential
equation (ODE):%
\begin{equation}
s(t)=L_{1}+\int_{t}^{T}F(s(r))dr,\;t\in\lbrack0,T].\label{u-ode}%
\end{equation}
Let $l(\cdot)$ be the minimal solution to the following ODE:
\begin{equation}
l(t)=-L_{1}-\int_{t}^{T}F(l(r))dr,\;t\in\lbrack0,T].\label{l-ode}%
\end{equation}
Moreover, set
\begin{equation}
t_{1}=T-\int_{-\infty}^{-L_{1}}\frac{1}{F(y)}dy,\ \ t_{2}=T-\int_{L_{1}%
}^{\infty}\frac{1}{F(y)}dy,\ \ t^{\ast}=t_{1}\vee t_{2}.\label{def-t}%
\end{equation}

\begin{assumption}
\label{assum-3} There exists a positive constant $\beta_{0}\in(0,1)$ such
that
\[
t^{\ast}<0
\]
and%
\begin{equation}
\lbrack s(0)\vee(-l(0))]L_{3}\leq1-\beta_{0}. \label{assum-cl}%
\end{equation}

\end{assumption}

\begin{remark}
When $L_{2}$ and $L_{3}$ are small enough, we derive that Assumption
\ref{assum-3} holds in \cite{Hu-JX}.
\end{remark}

We introduce the following notations: for $\psi=b,\sigma,g$ and
$\kappa=x,y,z$,%
\begin{equation}\label{new-nwe-1}%
\begin{array}
[c]{cl}%
\psi(s)= & \psi(s,\bar{X}^{t,x;\bar{u}}(s),\bar{Y}^{t,x;\bar{u}}(s),\bar
{Z}^{t,x;\bar{u}}(s),\bar{u}(s)),\\
\psi_{\kappa}(s)= & \psi_{\kappa}(s,\bar{X}^{t,x;\bar{u}}(s),\bar{Y}%
^{t,x;\bar{u}}(s),\bar{Z}^{t,x;\bar{u}}(s),\bar{u}(s)),\\
D\psi(s)= & D\psi(s,\bar{X}^{t,x;\bar{u}}(s),\bar{Y}^{t,x;\bar{u}}(s),\bar
{Z}^{t,x;\bar{u}}(s),\bar{u}(s)),\\
D^{2}\psi(s)= & D^{2}\psi(s,\bar{X}^{t,x;\bar{u}}(s),\bar{Y}^{t,x;\bar{u}%
}(s),\bar{Z}^{t,x;\bar{u}}(s),\bar{u}(s)),
\end{array}
\end{equation}
where $D\psi$ is the gradient of $\psi$ with respect to $x,y,z,$ and
$D^{2}\psi$ is the Hessian matrix of $\psi$ with respect to $x,y,z$.

The first-order adjoint equation is
\begin{equation}
\left\{
\begin{array}
[c]{rl}%
dp(s)= & -\left\{  g_{x}(s)+g_{y}(s)p(s)+g_{z}(s)K_{1}(s)+b_{x}(s)p(s)+b_{y}%
(s)p^{2}(s)\right. \\
& \left.  +b_{z}(s)K_{1}(s)p(s)+\sigma_{x}(s)q(s)+\sigma_{y}(s)p(s)q(s)+\sigma
_{z}(s)K_{1}(s)q(s)\right\}  ds\\
&+q(s)dB(s),\\
p(T)= & \phi_{x}(\bar{X}(T)),
\end{array}
\right.  \label{eq-p}%
\end{equation}
where
\begin{equation}
K_{1}(s)=(1-p(s)\sigma_{z}(s))^{-1}\left[  \sigma_{x}(s)p(s)+\sigma
_{y}(s)p^{2}(s)+q(s)\right]  . \label{new-k1}%
\end{equation}
The second-order adjoint equation is
\begin{equation}
\left\{
\begin{array}
[c]{rl}%
-dP(s)= & \left\{  P(s)\left[  (D\sigma(s)^{\intercal}\left(  1,p(s),K_{1}(s)\right)
^{\intercal})^{2}+2Db(s)^{\intercal}\left(  1,p(s),K_{1}(s)\right)  ^{\intercal
}\right.  \right.  \\
& \left.  +H_{y}(s)\right]  +2Q(s)D\sigma(s)^{\intercal}\left(  1,p(s),K_{1}%
(s)\right)  ^{\intercal}\\
& +\left(  1,p(s),K_{1}(s)\right)  D^{2}H(s)\left(  1,p(s),K_{1}(s)\right)
^{\intercal}\left.  +H_{z}(s)K_{2}(s)\right\}  ds\\
& -Q(s)dB(s),\\
P(T)= & \phi_{xx}(\bar{X}(T)),
\end{array}
\right.  \label{eq-P}%
\end{equation}
where
\[%
\begin{array}
[c]{ll}%
H(s,x,y,z,u,p,q)= & g(s,x,y,z,u)+pb(s,x,y,z,u)+q\sigma(s,x,y,z,u),
\end{array}
\]%
\begin{equation}%
\begin{array}
[c]{ll}%
&K_{2}(s)\\
&= (1-p(s)\sigma_{z}(s))^{-1}\left\{  p(s)\sigma_{y}(s)+2\left[
\sigma_{x}(s)+\sigma_{y}(s)p(s)+\sigma_{z}(s)K_{1}(s)\right]  \right\}  P(s)\\
&\ \  +(1-p(s)\sigma_{z}(s))^{-1}\left\{  Q(s)+p(s)\left(  1,p(s),K_{1}(s)\right)
D^{2}\sigma(s)\left(  1,p(s),K_{1}(s)\right)  ^{\intercal}\right\},
\end{array}
\label{new-k2}%
\end{equation}
$H(s)$, $DH(s)$ and $D^2H(s)$ are defined similarly in (\ref{new-nwe-1}).

Define
\begin{equation}%
\begin{array}
[c]{ll}%
&\mathcal{H}(s,x,y,z,u,p,q,P)\\
&= pb(s,x,y,z+\Delta(s),u)+q\sigma(s,x,y,z+\Delta
(s),u) +g(s,x,y,z+\Delta(s),u)\\
&\ \ +\frac{1}{2}P(\sigma(s,x,y,z+\Delta(s),u)-\sigma(s,\bar{x}(s),\bar
{y}(s),\bar{z}(s),\bar{u}(s)))^{2}\,
\end{array}
\label{def-H}%
\end{equation}
where $\Delta(s)$ is the solution to the following algebra equation
\begin{equation}
\begin{array}[c]{rl}
\Delta(s)=&p(s)\left[\sigma(s,\bar{X}^{t,x;\bar{u}}(s),\bar{Y}^{t,x;\bar{u}}%
(s),\bar{Z}^{t,x;\bar{u}}(s)+\Delta(s),u)\right.\\
&\left.-\sigma(s,\bar{X}^{t,x;\bar{u}%
}(s),\bar{Y}^{t,x;\bar{u}}(s),\bar{Z}^{t,x;\bar{u}}(s),\bar{u}(s))\right],\;s\in
\lbrack t,T].
\end{array} \label{def-delt}%
\end{equation}
Then, we have the following maximum principle.

\begin{theorem}
(See Theorem 3.18 in \cite{Hu-JX}) \label{Th-MP}Suppose that Assumptions \ref{assum-1},
\ref{assum-2} and \ref{assum-3} hold, and $q(\cdot)$ in (\ref{eq-p}) is
bounded. Then the following stochastic maximum principle holds:
\begin{equation}%
\begin{array}
[c]{l}%
\mathcal{H}(s,\bar{X}^{t,x;\bar{u}}(s),\bar{Y}^{t,x;\bar{u}}(s),\bar
{Z}^{t,x;\bar{u}}(s),u,p(s),q(s),P(s))\\
\geq\mathcal{H}(s,\bar{X}^{t,x;\bar{u}}(s),\bar{Y}^{t,x;\bar{u}}(s),\bar
{Z}^{t,x;\bar{u}}(s),\bar{u}(s),p(s),q(s),P(s)),\ \ \ \forall u\in
U\ a.e.,\ a.s..
\end{array}
\label{mp-1}%
\end{equation}

\end{theorem}

\begin{remark}
In the above theorem, if $\sigma(t,x,y,z,u)=A(t)z+\sigma_1(t,x,y,u)$, then we do not need
the assumption that $q(\cdot)$ is bounded.
\end{remark}

\section{Main results}
In the following, the constant $C>0$ will change from line to line for simplicity.
\subsection{Differentials in spatial variable}

In this subsection, we investigate the relationship between the MP and the DPP. We
first recall the notions of second-order super- and sub-jets in the spatial
variable $x$. For $w\in C([0,T]\times\mathbb{R})$ and $(t,\hat{x})\in
\lbrack0,T]\times\mathbb{R}$, define
\[
\left\{
\begin{array}
[c]{rll}%
D_{x}^{2,+}w(t,\hat{x}) & := & \{(p,P)\in\mathbb{R}\times\mathbb{R}:w(t,x)\leq
w\left(  t,\hat{x}\right)  +\left\langle p,x-\hat{x}\right\rangle \\
&  & +\frac{1}{2}(x-\hat{x})P(x-\hat{x})+o\left(  \left\vert x-\hat
{x}\right\vert ^{2}\right)  ,\text{ as }x\rightarrow\hat{x}\},\\
D_{x}^{2,-}w(t,\hat{x}) & := & \{(p,P)\in\mathbb{R}\times\mathbb{R}:w(t,x)\geq
w\left(  t,\hat{x}\right)  +\left\langle p,x-\hat{x}\right\rangle \\
&  & +\frac{1}{2}(x-\hat{x})P(x-\hat{x})+o\left(  \left\vert x-\hat
{x}\right\vert ^{2}\right)  ,\text{ as }x\rightarrow\hat{x}\}.
\end{array}
\right.
\]

\begin{theorem}
\label{th-visco-x} Suppose Assumptions \ref{assum-1}, \ref{assum-2} and
\ref{assum-3} hold. Let $\bar{u}(\cdot)$ be optimal for problem
\eqref{obje-eq}, and let $(p(\cdot),q(\cdot))$ and $(P(\cdot),Q(\cdot))$ $\in
L_{\mathbb{F}}^{\infty}(0,T;\mathbb{R})\times L_{\mathbb{F}}%
^{2,2}(0,T;\mathbb{R})$ be the solution to equation \eqref{eq-p} and
\eqref{eq-P} respectively. Furthermore, suppose that $q(\cdot)$ is bounded.
Then
\begin{equation}
\left\{
\begin{array}
[c]{l}%
\{p(s)\}\times\lbrack P(s),\infty)\subseteq D_{x}^{2,+}W(s,\bar{X}%
^{t,x;\bar{u}}(s)),\\
D_{x}^{2,-}W(s,\bar{X}^{t,x;\bar{u}}(s))\subseteq\{p(s)\}\times(-\infty
,P(s)],\text{ }\forall s\in\lbrack t,T],\text{ }P\text{-}a.s.
\end{array}
\right.  \label{main-relation}%
\end{equation}

\end{theorem}

\begin{proof}
The proof is divided into five steps.

\textbf{Step 1:} Variational equations.

For each fixed $s\in\lbrack t,T]$ and $x^{\prime}\in\mathbb{R}$, denote by
$(X^{s,x^{\prime};\bar{u}}(\cdot),Y^{s,x^{\prime};\bar{u}}(\cdot
),Z^{s,x^{\prime};\bar{u}}(\cdot))$ the solution to the following FBSDE:%
\begin{equation}
\left\{
\begin{array}
[c]{rl}%
dX^{s,x^{\prime};\bar{u}}(r)= & b(r,X^{s,x^{\prime};\bar{u}}(r),Y^{s,x^{\prime
};\bar{u}}(r),Z^{s,x^{\prime};\bar{u}}(r),\bar{u}(s))dr\\
&+\sigma
(r,X^{s,x^{\prime};\bar{u}}(r),Y^{s,x^{\prime};\bar{u}}(r),Z^{s,x^{\prime
};\bar{u}}(r),\bar{u}(s))dB(r),\\
dY^{s,x^{\prime};\bar{u}}(r)= & -g(r,X^{s,x^{\prime};\bar{u}}%
(r),Y^{s,x^{\prime};\bar{u}}(r),Z^{s,x^{\prime};\bar{u}}(r),\bar
{u}(s))dr+Z^{s,x^{\prime};\bar{u}}(r)dB(r),\\
X^{s,x^{\prime};\bar{u}}(s)= & x^{\prime},\ Y^{s,x^{\prime};\bar{u}}%
(T)=\phi(X^{s,x^{\prime};\bar{u}}(T)),\ \  r\in\lbrack s,T].
\end{array}
\right.  \label{nota-new-1}%
\end{equation}
Set $\bar{\Theta}(r):=  \left(  \bar{X}^{t,x;\bar{u}}(r),\bar{Y}^{t,x;\bar{u}%
}(r),\bar{Z}^{t,x;\bar{u}}(r)\right)$,
\begin{equation}%
\begin{array}
[c]{rl}
\hat{X}(r):= & X^{s,x^{\prime};\bar{u}}(r)-\bar{X}^{t,x;\bar{u}}(r),\ \hat{Y}(r):=  Y^{s,x^{\prime};\bar{u}}(r)-\bar{Y}^{t,x;\bar{u}}(r),\\
\hat{Z}(r):= & Z^{s,x^{\prime};\bar{u}}(r)-\bar{Z}^{t,x;\bar{u}}(r),\  \hat{\Theta}(r):= \left(  \hat{X}(r),\hat{Y}(r),\hat{Z}(r)\right)  .
\end{array}
\label{nota-new-2}%
\end{equation}
By Theorem 2.2 in \cite{Hu-JX}, for each $\beta\in\lbrack2,8]$, we have $P$\text{-}a.s.
\begin{equation}
\mathbb{E}\left[  \left.  \sup\limits_{r\in\lbrack s,T]}\left(  |\hat
{X}(r)|^{\beta}+|\hat{Y}(r)|^{\beta}\right)  +\left(  \int_{s}^{T}|\hat
{Z}(r)|^{2}dr\right)  ^{\frac{\beta}{2}}\right\vert \mathcal{F}_{s}%
^{t}\right]  \leq C\left\vert x^{\prime}-\bar{X}^{t,x;\bar{u}}(s)\right\vert
^{\beta}.\label{new-eq-11}%
\end{equation}
It is easy to check that $\left(  \hat{X}(\cdot),\hat{Y}(\cdot),\hat{Z}%
(\cdot)\right)  $ satisfies the following FBSDE:
\begin{equation}
\left\{
\begin{array}
[c]{rl}%
d\hat{X}(r)= & \left[  \hat{\Theta}(r)Db(r)+\varepsilon_{1}(r)\right]
dr+\left[  \hat{\Theta}(r)D\sigma(r)+\varepsilon_{2}(r)\right]  dB(r),\\
\hat{X}(s)= & x^{\prime}-\bar{X}^{t,x;\bar{u}}(s),\\
d\hat{Y}(r)= & -\left[  \hat{\Theta}(r)Dg(r)+\varepsilon_{3}(r)\right]
dr+\hat{Z}(r)dB(r),\text{ }r\in\lbrack s,T],\\
\hat{Y}(T)= & \phi_{x}(\bar{X}^{t,x;\bar{u}}(T))\hat{X}(T)+\varepsilon_{4}(T),
\end{array}
\right.  \label{eq-xy-hat}%
\end{equation}
where%
\begin{equation}%
\begin{array}
[c]{ll}%
\varepsilon_{1}(r)= & \left(  \tilde{b}_{x}(r)-b_{x}(r)\right)
\hat{X}(r)+\left(  \tilde{b}_{y}(r)-b_{y}(r)\right)  \hat
{Y}(r)+\left(  \tilde{b}_{z}(r)-b_{z}(r)\right)  \hat{Z}(r),\\
\varepsilon_{2}(r)= & \left(  \tilde{\sigma}_{x}(r)-\sigma
_{x}(r)\right)  \hat{X}(r)+\left(  \tilde{\sigma}_{y}%
(r)-\sigma_{y}(r)\right)  \hat{Y}(r)+\left(  \tilde{\sigma}_{z}(r)-\sigma_{z}(r)\right)  \hat{Z}(r),\\
\varepsilon_{3}(r)= & \left(  \tilde{g}_{x}(r)-g_{x}(r)\right)
\hat{X}(r)+\left(  \tilde{g}_{y}(r)-g_{y}(r)\right)  \hat
{Y}(r)+\left(  \tilde{g}_{z}(r)-g_{z}(r)\right)  \hat{Z}(r),\\
\varepsilon_{4}(T)= & [\tilde{\phi}_{x}(T)-\phi_{x}(\bar{X}^{t,x;\bar{u}}(T))]\hat
{X}(T),\\
\tilde{\psi}_{\kappa}(r)= &
%TCIMACRO{\dint \nolimits_{0}^{1}}%
%BeginExpansion
{\displaystyle\int\nolimits_{0}^{1}}
%EndExpansion
 \psi_{\kappa}(r,\bar{\Theta}(r)+\lambda\hat{\Theta}(r),\bar
{u}(r))  d\lambda \text{ for }\psi=b,\sigma
,g,\phi  \text{ and }\kappa=x,y,z.
\end{array}
\label{new-eqqw-11112}%
\end{equation}

\textbf{Step 2:} Estimates of the remainder terms of FBSDE.

By Assumption \ref{assum-2}, we derive that, for $i=1,2,3,$
\[
|\varepsilon_{i}(r)|\leq C\left(  |\hat{X}(r)|^{2}+|\hat{Y}(r)|^{2}+|\hat
{Z}(r)|^{2}\right)  \text{ and }|\varepsilon_{4}(T)|\leq C|\hat{X}(T)|^{2}.
\]
Then, by (\ref{new-eq-11}), we obtain that for each $\beta\in\lbrack2,4]$,
\begin{equation}%
\begin{array}
[c]{rl}%
\mathbb{E}\left[  \left.  \left(  \int_{s}^{T}|\varepsilon_{i}(r)|dr\right)
^{\beta}\right\vert \mathcal{F}_{s}^{t}\right]  & \leq C \left\vert x^{\prime}%
-\bar{X}^{t,x;\bar{u}}(s)\right\vert ^{2\beta},\text{ }i=1,2,3,\\
\mathbb{E}\left[  \left.  |\varepsilon_{4}(T)|^{\beta}\right\vert
\mathcal{F}_{s}^{t}\right]  & \leq C \left\vert x^{\prime}-\bar{X}^{t,x;\bar{u}%
}(s)\right\vert ^{2\beta}.
\end{array}
\label{est-epsi-1-4}%
\end{equation}

\textbf{Step 3}: Relationship between $\hat{X}(\cdot)$ and $(\hat{Y}%
(\cdot),\hat{Z}(\cdot))$.

By Theorem 5.3 in \cite{Hu-JX}, we get%
\begin{equation}%
\begin{array}
[c]{rl}%
\hat{Y}(r)= & p(r)\hat{X}(r)+\varphi(r),\\
\hat{Z}(r)= & K_{1}(r)\hat{X}(r)+(1-p(r)\sigma_{z}(r))^{-1}[p(r)\sigma_y(r)\varphi(r)+p(r)\varepsilon_2(r)+\nu(r)],
\end{array}
\label{relation-hat}%
\end{equation}
where $p(\cdot)$ is the solution to first-order adjoint equation (\ref{eq-p}),
and $(\varphi(\cdot),\nu(\cdot))$ is the solution to the following linear
BSDE:
\begin{equation}
\left\{
\begin{array}
[c]{cl}%
d\varphi(r)= & -\left[  A(r)\varphi(r)+C(r)\nu(r)+p(r)\varepsilon
_{1}(r)+q(r)\varepsilon_{2}(r)+\varepsilon_{3}(r)\right.\\
&\left.+H_{z}(r)(1-p(r)\sigma
_{z}(r))^{-1}p(r)\varepsilon_{2}(r)\right]  dr+\nu(r)dB(r),\\
\varphi(T)= & \varepsilon_{4}(T),
\end{array}
\right.  \label{eq-varphi}%
\end{equation}
\[%
\begin{array}
[c]{rl}%
A(r)= & p(r)b_{y}(r)+q(r)\sigma_{y}(r)+g_{y}(r)+(1-p(r)\sigma_{z}%
(r))^{-1}\sigma_{y}(r)p(r)H_{z}(r),\\
C(r)= & (1-p(r)\sigma_{z}(r))^{-1}H_{z}(r),\\
H_{z}(r)= & p(r)b_{z}(r)+q(r)\sigma_{z}(r)+g_{z}(r).
\end{array}
\]
Thus we can write $\hat{\Theta}(r)$ as%
\begin{equation}\label{def-theta-hat}
\hat{\Theta}(r)=(1,p(r),K_{1}(r))\hat{X}(r)+\hat{L}(r),
\end{equation}
where $\hat{L}(r):=(0,\varphi(r),(1-p(r)\sigma_{z}(r))^{-1}[p(r)\sigma_y(r)\varphi(r)+p(r)\varepsilon_2(r)+\nu(r)]).$

It follows from Theorem 3.6 in \cite{Hu-JX} that%
\begin{equation}\label{p-less}
|p(r)|\leq s(0)\vee(-l(0)) \text{ for }r\in\lbrack s,T].
\end{equation}
By relations (\ref{new-eqqw-11112}) and (\ref{relation-hat}), we get
\begin{equation}%
\begin{array}
[c]{ll}%
&\varepsilon_{2}\left(  r\right) \\
& =\left(  \tilde{\sigma}_{x}
(r)-\sigma_{x}(r)\right)  \hat{X}(r)+\left(  \tilde{\sigma}_{y}
(r)-\sigma_{y}(r)\right)  \left(  p(r)\hat{X}\left(  r\right)  +\varphi
\left(  r\right)  \right) +\left(  \tilde{\sigma}_{z}(r)-\sigma
_{z}(r)\right)\\
&\ \ \ \cdot \left[ K_1(r)\hat{X}(r)+ \left(  1-p\left(  r\right)  \sigma_{z}\right)
^{-1}\left(  p\left(  r\right)  \sigma_{y}\left(  r\right)  \varphi\left(
r\right)  +p\left(  r\right)  \varepsilon_{2}\left(  r\right)  +\nu\left(
r\right)  \right)  \right]  .
\end{array}
\label{eq-epsil2}%
\end{equation}
Thus, we have
\begin{equation}
\begin{array}
[c]{l}%
\varepsilon_{2}\left(  r\right)  =\frac{1-p\left(  r\right)  \sigma_{z}%
(r)}{1-p\left(  r\right)  \tilde{\sigma}_{z}(r)}\left\{  \left[
\left(  \tilde{\sigma}_{x}(r)-\sigma_{x}(r)\right)  \hat
{X}(r)+\left(  \tilde{\sigma}_{y}(r)-\sigma_{y}(r)\right)
\left(  p(r)\hat{X}\left(  r\right)  +\varphi\left(  r\right)  \right)
\right]  \right.  \\
\text{ \ \ \ \ \ \ \ \ \ }\left.  +\left(  \tilde{\sigma}_{z}
(r)-\sigma_{z}(r)\right)  \left[ K_1(r)\hat{X}(r)+  \left(  p\left(  r\right)  \sigma
_{y}\left(  r\right)  \varphi\left(  r\right)  +\nu\left(  r\right)  \right)
\right]  \right\}  .
\end{array}
\label{eq-epsilon2}
\end{equation}
By Assumption \ref{assum-3} and (\ref{p-less}),  one has $\left\vert \frac{1-p\left(  r\right)  \sigma_{z}(r)}{1-p\left(
r\right)  \tilde{\sigma}_{z}(r)}\right\vert \leq\beta_{0}^{-1}$.

\textbf{Step 4:} Variation of $\varphi$.

By (\ref{est-epsi-1-4}) and the estimate of BSDE for (\ref{eq-varphi}), we obtain that, for each $\beta\in
\lbrack2,4],$
\begin{equation}%
\begin{array}
[c]{l}%
\mathbb{E}\left[  \left.  \sup\limits_{r\in\lbrack s,T]}|\varphi(r)|^{\beta
}+\left(  \int_{s}^{T}|\nu(r)|^{2}dr\right)  ^{\frac{\beta}{2}}\right\vert
\mathcal{F}_{s}^{t}\right] \\
\leq C\mathbb{E}\left[  \left.  |\varepsilon_{4}(T)|^{\beta}+\left(  \int%
_{s}^{T}(|\varepsilon_{1}(r)|+|\varepsilon_{2}(r)|+|\varepsilon_{3}%
(r)|)dr\right)  ^{\beta}\right\vert \mathcal{F}_{s}^{t}\right] \\
\leq C\left\vert x^{\prime}-\bar{X}^{t,x;\bar{u}}(s)\right\vert ^{2\beta
},P\text{-}a.s..
\end{array}
\label{est-phi}%
\end{equation}
In the followings, we want to prove%
\begin{equation}
\varphi(s)-\frac{1}{2}P(r)(\hat{X}(r))^{2} =o(\left\vert x^{\prime}-\bar
{X}^{t,x;\bar{u}}(s)\right\vert ^{2}),\text{ }P \text{-}a.s.. \label{est-phi-hat}%
\end{equation}
Define%
\begin{align}
\tilde{\varphi}(r)  &  =\frac{1}{2}P(r)(\hat{X}(r))^{2};\\
\tilde{\nu}(r)  &  =P(r)\hat{X}(r)(\hat{\Theta}(r)D\sigma(r)+\varepsilon
_{2}(r))+\frac{1}{2}Q(r)\left(  \hat{X}(r)\right)  ^{2}.
\end{align}
Applying It\^{o}'s formula to $\frac{1}{2}P(r)(\hat{X}(r))^{2}$, by (\ref{def-theta-hat}), we obtain
that $\left(  \tilde{\varphi}(r),\tilde{\nu}(r)\right)  $ satisfies the
following BSDE:%
\begin{equation}
\left\{
\begin{array}
[c]{ll}%
d\tilde{\varphi}(r) =&\left\{ P(r)\left[ \left(  \hat{L}(r)Db(r)+\varepsilon
_{1}(r)\right)  \hat{X}(r) +\frac{1}{2}(\hat{L}(r)D\sigma(r)+\varepsilon_{2}(r))^{2}
\right.\right.\\
&\left.+(1,p(r),K_{1}(r))D\sigma(r)\hat{X}(r)(\hat
{L}(r)D\sigma(r)+\varepsilon_{2}(r))-\frac{1}{2}H_y(r)\hat{X}(r)^2 \right]\\
&-\frac{1}{2}\left[(1,p(r),K_1(r))D^2H(s)(1,p(r),K_1(r))^{\intercal}+H_z(s)K_2(s)\right]\hat{X}(r)^2\\
&\left.+Q(r)(\hat{L}(r)D\sigma(r)+\varepsilon_{2}(r))\hat{X}(r) \right\}dr+\tilde{\nu}(r)dB(r),\\
\tilde{\varphi}(T)=& \frac{1}{2}\phi_{xx}(\bar{X}^{t,x;\bar{u}}(T))\left(
\hat{X}(T)\right)  ^{2}.
\end{array}
\right.  \label{eq-PX2}%
\end{equation}
Set
\[
\hat{\varphi}(r)=\varphi(r)-\tilde{\varphi}(r),\ \hat{\nu}(r)=\nu\left(
r\right)  -\tilde{\nu}(r).
\]

Replace $\varepsilon_{1}(r)$ by $\frac{1}{2}\hat{\Theta}(r)D^{2}%
b(r)\hat{\Theta}(r)^{\intercal}+\varepsilon_{5}(r)$, $\varepsilon_{2}(r)$ by $\frac{1}%
{2}\hat{\Theta}(r)D^{2}\sigma(r)\hat{\Theta}(r)^{\intercal}+\varepsilon_{6}(r)$,
$\varepsilon_{3}(r)$ by $\frac{1}{2}\hat{\Theta}(r)D^{2}g(r)\hat{\Theta
}(r)^{\intercal}+\varepsilon_{7}(r)$ and $\varepsilon_{4}(T)$ by $\frac{1}{2}\phi
_{xx}(\bar{X}^{t,x;\bar{u}}(T))\left(  \hat{X}(T)\right)  ^{2}+\varepsilon
_{8}(T)$ in (\ref{eq-varphi}), where
\[%
\begin{array}
[c]{ll}%
\varepsilon_{5}(r)= & \hat{\Theta}(r)\int_{0}^{1}\int_{0}^{1}%
\lambda\left[  D^{2}b(r,\bar{\Theta}^{t,x;\bar{u}}(r)+\theta\lambda\hat
{\Theta}(r),\bar{u}(r))-D^{2}b(r)\right]  d\lambda d\theta\hat{\Theta}(r)^{\intercal},\\
\varepsilon_{6}(r)= & \hat{\Theta}(r)\int_{0}^{1}\int_{0}^{1}%
\lambda\left[  D^{2}\sigma(r,\bar{\Theta}^{t,x;\bar{u}}(r)+\theta\lambda
\hat{\Theta}(r),\bar{u}(r))-D^{2}\sigma(r)\right]  d\lambda d\theta\hat
{\Theta}(r)^{\intercal},\\
\varepsilon_{7}(r)= & \hat{\Theta}(r)\int_{0}^{1}\int_{0}^{1}%
\lambda\left[  D^{2}g(r,\bar{\Theta}^{t,x;\bar{u}}(r)+\theta\lambda\hat
{\Theta}(r),\bar{u}(r))-D^{2}g(r)\right]  d\lambda d\theta\hat{\Theta}(r)^{\intercal},\\
\varepsilon_{8}(T)= & \int_{0}^{1}\int_{0}^{1}\lambda\left[  \phi_{xx}(\bar
{X}^{t,x;\bar{u}}(r)+\theta\lambda\hat{X}(T))-\phi_{xx}(\bar{X}^{t,x;\bar{u}%
}(T))\right]  d\lambda d\theta\left(  \hat{X}(T)\right)  ^{2}.
\end{array}
\]
By (\ref{new-k2}), one can verify that $\left(  \hat{\varphi}(\cdot),\hat{\nu}%
(\cdot)\right)  $ satisfies the following linear BSDE
\begin{equation}\label{eq-510}
\left\{
\begin{array}
[c]{cl}%
d\hat{\varphi}(r)= & -\left[  A(r)\hat{\varphi}(r)+C(r)\hat{\nu}%
(r)+I(r)\right]  dr+\hat{\nu}(r)dB\left(  r\right)  ,\\
\hat{\varphi}(T)= & \varepsilon_{8}(T),
\end{array}
\right.
\end{equation}
where
\begin{equation}\label{def-I}
\begin{array}
[c]{ll}%
\mathrm{I}(r)=&C(r)P(r)(\hat{L}(r)D\sigma(r)+\varepsilon_2(r))\hat{X}(r)+p(r)\varepsilon_{5}(r)+\varepsilon_{7}(r)\\
&+p(r)(1,p(r),K_1(r))D^2b(r)\hat{L}(r)^{\intercal}\hat{X}(r)+\frac{1}{2}p(r)\hat{L}(r)D^2b(r)\hat{L}(r)^{\intercal}\\
&+\left[  q(r)+H_{z}(r)(1-p(r)\sigma_{z}(r))^{-1}p(r)\right] \\
& \ \ \ \cdot \left[
\frac{1}{2}\hat{L}(r)D^{2}\sigma(r)\hat{L}\left(  r\right)  ^{\intercal}%
+(1,p(r),K_{1}(r))D^{2}\sigma(r)\hat{L}\left(  r\right)  ^{\intercal}\hat
{X}(r)+\varepsilon_{6}(r)\right] \\
&+(1,p(r),K_1(r))D^2g(r)\hat{L}(r)^{\intercal}\hat{X}(r)+\frac{1}{2}\hat{L}(r)D^2g(r)\hat{L}(r)^{\intercal}\\
&+P(r)\left[  \left(  \hat{L}(r)Db(r)+\varepsilon
_{1}(r)\right)  \hat{X}(r)+\frac{1}{2}(\hat{L}(r)D\sigma(r)+\varepsilon_{2}%
(r))^{2}\right.\\
&\left.+(1,p(r),K_{1}(r))D\sigma(r)\left(  \hat
{L}(r)D\sigma(r)+\varepsilon_{2}(r)\right)  \hat{X}(r)\right]\\
&+Q(r)\left( \hat{L}(r)D\sigma(r)+\varepsilon_{2}(r)\right)  \hat{X}(r).
\end{array}
\end{equation}
%By the solution to linear BSDE, we have
%\begin{equation}
%\hat{\varphi}\left(  s\right)  =\mathbb{E}\left[  \left.  \gamma\left(
%T\right)  \varepsilon_{8}\left(  T\right)  +\int_{s}^{T}\gamma\left(
%r\right)  I\left(  r\right)  dr\right\vert \mathcal{F}_{s}^{t}\right],
%\end{equation}
%where $\gamma(\cdot)$ satisfies the following SDE
%\begin{equation}
%\left\{
%\begin{array}
%[c]{l}%
%d\gamma\left(  r\right)  =A\left(  r\right)  \gamma\left(  r\right)
%dr+C\left(  r\right)  \gamma\left(  r\right)  dB\left(  r\right)  ,\\
%\gamma\left(  s\right)  =1.
%\end{array}
%\right.
%\end{equation}
Note that $A(\cdot)$ and $C(\cdot)$ are bounded. Then by the standard estimate of BSDE, we obtain that
\begin{equation}
\left\vert \hat{\varphi}(s)\right\vert ^{2}\leq C\mathbb{E}\left[  \left.
\left\vert \varepsilon_{8}(T)\right\vert ^{2}+\left(  \int_{s}^{T}%
|\mathrm{I}(r)|dr\right)  ^{2}\right\vert \mathcal{F}_{s}^{t}\right]  .
\end{equation}
Since $q(\cdot)$ is bounded, one can verify that $P(\cdot)$ is bounded. By (\ref{new-eqqw-11112}), (\ref{def-theta-hat}) and (\ref{eq-epsilon2}), it is easy to check that
\begin{equation}%
\begin{array}
[c]{rl}%
\left\vert \mathrm{I} \left(  r\right)  \right\vert \leq & C\left[  \left(  1+\left\vert
Q\left(  r\right)  \right\vert \right)  \left(  \rho\left(  r\right)
\left\vert \hat{X}(r)\right\vert ^{2}+\left\vert \hat{X}(r)\varphi\left(
r\right)  \right\vert +\left\vert \hat{X}(r)\nu\left(  r\right)  \right\vert
\right)  \right.  \\
& \left.  +\left\vert \varphi\left(  r\right)  \right\vert ^{2}+\left\vert
\nu\left(  r\right)  \right\vert ^{2}\right],
\end{array}
\end{equation}
where $C$ is a constant and
\begin{equation}%
\begin{array}
[c]{rl}%
\rho\left(  r\right)  = & \sum_{i=1}^{3}\displaystyle\int_{0}^{1}\int_{0}^{1}\lambda\left|
D^{2}\psi_{i}(r,\bar{\Theta}(r)+\theta\lambda\hat{\Theta}(r),\bar{u}%
(r))-D^{2}\psi_{i}(r)\right|  d\lambda d\theta\\
& +\sum_{i=1}^{2}\displaystyle\int_{0}^{1}\left|  D\psi_{i}(r,\bar{\Theta}(r)+\theta
\hat{\Theta}(r),\bar{u}(r))-D\psi_{i}(r)\right|  d\theta
\end{array}
\label{def-rho}
\end{equation}
for $\psi_{1}=b$, $\psi_{2}=\sigma$, $\psi_{3}=g$.

Next, we estimate term by term.
\[%
\begin{array}
[c]{l}%
\mathbb{E}\left[  \left.  \left\vert \varepsilon_{8}(T)\right\vert
^{2}\right\vert \mathcal{F}_{s}^{t}\right] \\
\leq\left\{  \mathbb{E}\left[  \left.  \left\vert \hat{X}(T)\right\vert
^{8}\right\vert \mathcal{F}_{s}^{t}\right]  \right\}  ^{\frac{1}{2}}\\
\ \ \cdot\left\{
\mathbb{E}\left[  \left.  \left\vert \int_{0}^{1}\int_{0}^{1}\lambda\left[
\phi_{xx}(\bar{X}^{t,x;\bar{u}}(r)+\theta\lambda\hat{X}(T))-\phi_{xx}(\bar
{X}^{t,x;\bar{u}}(T))\right]  d\lambda d\theta\right\vert ^{4}\right\vert
\mathcal{F}_{s}^{t}\right]  \right\}  ^{\frac{1}{2}}\\
= o(\left\vert x^{\prime}-\bar{X}^{t,x;\bar{u}}(s)\right\vert ^{4});
\end{array}
\]%

 \[%
\begin{array}
[c]{l}%
\mathbb{E}\left[  \left.  \left(  \int_{s}^{T}\left(  1+\left\vert Q\left(
r\right)  \right\vert \right)  \rho\left(  r\right)  |\hat{X}(r)|^{2}%
dr\right)  ^{2}\right\vert \mathcal{F}_{s}^{t}\right]  \\
\leq C\mathbb{E}\left[  \left.  \underset{s\leq r\leq T}{\sup}\left\vert
\hat{X}(r)\right\vert ^{4}\int_{s}^{T}\left(  1+\left\vert Q\left(  r\right)
\right\vert \right)  ^{2}dr\int_{s}^{T}\rho\left(  r\right)  ^{2}dr\right\vert
\mathcal{F}_{s}^{t}\right]  \\
=o(\left\vert x^{\prime}-\bar{X}^{t,x;\bar{u}}(s)\right\vert ^{4});
\end{array}
\]%
\begin{equation}
\begin{array}
[c]{l}%
\mathbb{E}\left[  \left.  \left(  \int_{s}^{T}\left(  1+\left\vert Q\left(
r\right)  \right\vert \right)  \left\vert \hat{X}(r)\nu\left(  r\right)
\right\vert dr\right)  ^{2}\right\vert \mathcal{F}_{s}^{t}\right]  \\
\leq C\mathbb{E}\left[  \left.  \underset{s\leq r\leq T}{\sup}\left\vert
\hat{X}(r)\right\vert ^{2}\int_{s}^{T}\left(  1+\left\vert Q\left(  r\right)
\right\vert \right)  ^{2}dr\int_{s}^{T}\left\vert \nu\left(  r\right)
\right\vert ^{2}dr\right\vert \mathcal{F}_{s}^{t}\right]  \\
\leq C\left\{  \mathbb{E}\left[  \left.  \underset{s\leq r\leq T}{\sup
}\left\vert \hat{X}(r)\right\vert ^{4}\right\vert \mathcal{F}_{s}^{t}\right]
\right\}  ^{\frac{1}{2}}\left\{  \mathbb{E}\left[  \left.  \left(  \int%
_{s}^{T}\left\vert \nu\left(  r\right)  \right\vert ^{2}dr\right)
^{4}\right\vert \mathcal{F}_{s}^{t}\right]  \right\}  ^{\frac{1}{4}}\\
\ \ \cdot\left\{
\mathbb{E}\left[  \left.  \left(  \int_{s}^{T}\left(  1+\left\vert Q\left(
r\right)  \right\vert \right)  ^{2}dr\right)  ^{4}\right\vert \mathcal{F}%
_{s}^{t}\right]  \right\}  ^{\frac{1}{4}}\\
=o(\left\vert x^{\prime}-\bar{X}^{t,x;\bar{u}}(s)\right\vert ^{4});
\end{array}
\label{eq-est059}
\end{equation}
$$\mathbb{E}\left[\left(\int_{s}^{T}\left\vert \varphi\left(  r\right)  \right\vert ^{2}+\left\vert
\nu\left(  r\right)  \right\vert ^{2}dr\right)^2\right]=o(|x^{\prime}-\bar{X}^{t,x;\bar{u}}(s)|^4).$$
The estimate for $\mathbb{E}\left[  \left.  \left(  \int_{s}^{T}\left(  1+\left\vert Q\left(
r\right)  \right\vert \right)  \left\vert \hat{X}(r)\varphi\left(  r\right)
\right\vert dr\right)  ^{2}\right\vert \mathcal{F}_{s}^{t}\right]$ is similar to (\ref{eq-est059}). Thus, we obtain
\[
\left\vert
\hat{\varphi}(s)\right\vert =o(\left\vert x^{\prime}-\bar{X}^{t,x;\bar{u}%
}(s)\right\vert ^{2}), \  P\text{-}a.s..
\]

\textbf{Step 5}: Completion of the proof.

Since the set of all rational $x'\in \mathbb{R}$ is countable, we can find a subset
$\Omega_{0}\subseteq\Omega$ with $P(\Omega_{0})=1$ such that for any
$\omega_{0}\in\Omega_{0}$,
\[
\left\{
\begin{array}
[c]{l}%
W(s,\bar{X}^{t,x;\bar{u}}(s,\omega_{0})=\bar{Y}^{t,x;\bar{u}}(s,\omega
_{0}),\text{ (\ref{new-eq-11}), (\ref{est-epsi-1-4}), (\ref{relation-hat}),
(\ref{est-phi}), (\ref{est-phi-hat}) are satisfied}\\
\text{for any rational }x^{\prime}\text{,}\ (\Omega,\mathcal{F},P(\cdot|\mathcal{F}_{s}^{t})\left(  \omega_{0}\right)
,B(\cdot)-B(s);u(\cdot))|_{[s,T]}\in \mathcal{U}^{w}[s,T],\text{and
} \\
\underset{s\leq r\leq T}{\sup}\left[  \left\vert p\left(  r,\omega
_{0}\right)  \right\vert +\left\vert P(r,\omega_{0})\right\vert \right]
<\infty .
\end{array}
\right.
\]
The first relation of the above is obtained by the DPP (see \cite{Hu-JX-DPP}).
Let $\omega_{0}\in\Omega_{0}$ be fixed, and then for any rational number
$x^{\prime}$,%
\begin{equation}
\left\vert \hat{\varphi}(s,\omega_{0})\right\vert =o(\left\vert x^{\prime
}-\bar{X}^{t,x;\bar{u}}(s,\omega_{0})\right\vert ^{2}),\text{ for all }%
s\in\lbrack t,T].
\end{equation}
By the definition of $\hat{\varphi}(s)$, we get for each $s\in\lbrack t,T]$,
\[%
\begin{array}
[c]{l}%
Y^{s,x^{\prime};\bar{u}}(s,\omega_{0})-\bar{Y}^{t,x;\bar{u}}(s,\omega_{0})\\
=p(s,\omega_{0})\hat{X}(s,\omega_{0})+\frac{1}{2}P(s,\omega_{0})\hat
{X}(s,\omega_{0})^{2}+o(\left\vert x^{\prime}-\bar{X}^{t,x;\bar{u}}%
(s,\omega_{0})\right\vert ^{2})\\
=p(s,\omega_{0})(x^{\prime}-\bar{X}^{t,x;\bar{u}}(s,\omega_{0}))+\frac{1}%
{2}P(s,\omega_{0})\left(  x^{\prime}-\bar{X}^{t,x;\bar{u}}(s,\omega
_{0})\right)  ^{2}\\
\ \ +o(\left\vert x^{\prime}-\bar{X}^{t,x;\bar{u}}(s,\omega
_{0})\right\vert ^{2}).
\end{array}
\]
Thus, for each $s\in\lbrack t,T]$,%
\begin{equation}%
\begin{array}
[c]{l}%
W(s,x^{\prime})-W(s,\bar{X}^{t,x;\bar{u}}(s,\omega_{0}))\\
\leq Y^{s,x^{\prime};\bar{u}}(s,\omega_{0})-\bar{Y}^{t,x;\bar{u}}(s,\omega
_{0})\\
=p(s,\omega_{0})(x^{\prime}-\bar{X}^{t,x;\bar{u}}(s,\omega_{0}))+\frac{1}%
{2}P(s,\omega_{0})\left(  x^{\prime}-\bar{X}^{t,x;\bar{u}}(s,\omega
_{0})\right)  ^{2}\\
\ \ +o(\left\vert x^{\prime}-\bar{X}^{t,x;\bar{u}}(s,\omega
_{0})\right\vert ^{2}).
\end{array}
\label{inequality-final}%
\end{equation}
Note that the term $o(\left\vert x^{\prime}-\bar{X}^{t,x;\bar{u}}(s,\omega
_{0})\right\vert ^{2})$ in the above depends only on the size of $\left\vert x^{\prime}-\bar{X}^{t,x;\bar{u}}(s,\omega
_{0})\right\vert$ and it is independent of $x'$. Therefore, by the continuity of $W(s,\cdot)$, we can easily obtain that
(\ref{inequality-final}) holds for all $x^\prime\in\mathbb{R}$. By the
definition of super-jets, we have
\[
\{p(s)\}\times\lbrack P(s),\infty)\subseteq D_{x}^{2,+}W(s,X^{t,x;\bar{u}%
}(s)).
\]
Now we prove that
\[
D_{x}^{2,-}W(s,X^{t,x;\bar{u}}(s))\subseteq\{p(s)\}\times(-\infty,P(s)].
\]
Fix an $\omega\in\Omega$ such that (\ref{inequality-final}) holds for all
$x^{\prime}\in\mathbb{R}$. For any $$(\hat{p},\hat{P})\in D_{x}^{2,-}%
V(s,\bar{X}^{t,x;\bar{u}}(s)),$$ by definition of sub-jets, we deduce%
\[%
\begin{array}
[c]{rl}%
0&\leq\underset{x^{\prime}\rightarrow \bar{X}^{t,x;\bar{u}}(s)}{\lim\inf}\left\{
\frac{W(s,x^{\prime})-W(s,\bar{X}^{t,x;\bar{u}}(s))-\hat{p}\left(  x^{\prime
}-\bar{X}^{t,x;\bar{u}}(s)\right)  -\frac{1}{2}\hat{P}\left(  x^{\prime}%
-\bar{X}^{t,x;\bar{u}}(s)\right)  ^{2}}{|x^{\prime}-\bar{X}^{t,x;\bar{u}%
}(s)|^{2}}\right\} \\
&\leq\underset{x^{\prime}\rightarrow \bar{X}^{t,x;\bar{u}}(s)}{\lim\inf}\left\{
\frac{\left(  p(s)-\hat{p}\right)  \left(  x^{\prime}-\bar{X}^{t,x;\bar{u}%
}(s)\right)  +\frac{1}{2}(P(s)-\hat{P})\left(  x^{\prime}-\bar{X}^{t,x;\bar
{u}}(s)\right)  ^{2}}{|x^{\prime}-\bar{X}^{t,x;\bar{u}}(s)|^{2}}\right\}  .
\end{array}
\]
Then it is necessary that
\[
\hat{p}=p(s),\hat{P}\leq P(s),\text{ }\forall s\in\lbrack t,T],\ P\text{-}a.s..
\]
This completes the proof.
\end{proof}

\subsection{Differential in time variable}

Let us recall the notions of right super-and sub-jets in the time variable
$t$. For $w\in C([0,T]\times\mathbb{R})$ and $(\hat{t},\hat{x})\in
\lbrack0,T)\times\mathbb{R}$, define
\[
\left\{
\begin{array}
[c]{rll}%
D_{t+}^{1,+}w(\hat{t},\hat{x}) & := & \{q\in\mathbb{R}:w\left(  t,\hat
{x}\right)  \leq w(\hat{t},\hat{x})+q(t-\hat{t})+o\left(  \left\vert t-\hat
{t}\right\vert \right)  \text{ as }t\downarrow\hat{t},\\
D_{t+}^{1,-}w(\hat{t},\hat{x}) & := & \{q\in\mathbb{R}:w\left(  t,\hat
{x}\right)  \geq w(\hat{t},\hat{x})+q(t-\hat{t})+o\left(  \left\vert t-\hat
{t}\right\vert \right)  \text{ as }t\downarrow\hat{t}.
\end{array}
\right.
\]

\begin{theorem}
\label{th-visco-t}Suppose the same assumptions as in Theorem \ref{th-visco-x}.
Then, for each $s\in\lbrack t,T]$,
\[
\left\{
\begin{array}
[c]{l}%
\lbrack\mathcal{H}_{1}(s,\bar{X}^{t,x;\bar{u}}(s),\bar{Y}^{t,x;\bar{u}%
}(s),\bar{Z}^{t,x;\bar{u}}(s)),\infty)\subseteq D_{t+}^{1,+}W(s,X^{t,x;\bar
{u}}(s)),\\
D_{t+}^{1,-}W(s,X^{t,x;\bar{u}}(s))\subseteq(-\infty,\mathcal{H}_{1}(s,\bar
{X}^{t,x;\bar{u}}(s),\bar{Y}^{t,x;\bar{u}}(s),\bar{Z}^{t,x;\bar{u}%
}(s))],\text{ }P\text{-}a.s.,
\end{array}
\right.
\]
where%
\[
\begin{array}
[c]{l}%
\mathcal{H}_{1}(s,\bar{X}^{t,x;\bar{u}}(s),\bar{Y}^{t,x;\bar{u}}(s),\bar
{Z}^{t,x;\bar{u}}(s))\\=-\mathcal{H}(s,\bar{X}^{t,x;\bar{u}}(s),\bar
{Y}^{t,x;\bar{u}}(s),\bar{Z}^{t,x;\bar{u}}(s),\bar{u}%
(s),p(t),q(t),P(t))+P(s)\sigma\left(  s\right)  ^{2}.
\end{array}
\]

\end{theorem}

\begin{proof}
The proof is divided into two steps.

\textbf{Step 1:} Variations and estimations for FBSDE.

For each $s\in(t,T)$, take $\tau\in(s,T].$ Denote by
\[
\Theta^{\tau,\bar{X}^{t,x;\bar{u}}(s);\bar{u}}(\cdot)=(X^{\tau,\bar
{X}^{t,x;\bar{u}}(s);\bar{u}}(\cdot),Y^{\tau,\bar{X}^{t,x;\bar{u}}(s);\bar{u}%
}(\cdot),Z^{\tau,\bar{X}^{t,x;\bar{u}}(s);\bar{u}}(\cdot))
\]
the solution to the following FBSDE on $[\tau,T]:$%
\[
\left\{
\begin{array}
[c]{rl}%
X^{\tau,X^{t,x;\bar{u}}(s);\bar{u}}(r)=&\bar{X}^{t,x;\bar{u}}(s)+\int_{\tau
}^{r}b(\alpha,\Theta^{\tau,\bar{X}^{t,x;\bar{u}}(s);\bar{u}}(\alpha),\bar
{u}(\alpha))d\alpha\\
&+\int_{\tau}^{r}\sigma(\alpha,\Theta^{\tau,\bar
{X}^{t,x;\bar{u}}(s);\bar{u}}(\alpha),\bar{u}(\alpha))dB(\alpha),\\
Y^{\tau,X^{t,x;\bar{u}}(s);\bar{u}}(r)=&\phi(X^{\tau,X^{t,x;\bar{u}}(s);\bar
{u}}(T))+\int_{r}^{T}g(\alpha,\Theta^{\tau,\bar{X}^{t,x;\bar{u}}(s);\bar{u}%
}(\alpha),\bar{u}(\alpha))d\alpha\\
&-\int_{r}^{T}Z^{\tau,\bar{X}^{t,x;\bar{u}%
}(s);\bar{u}}(\alpha)dB(\alpha).
\end{array}
\right.
\]
For $r\in\lbrack\tau,T]$, set
\[%
\begin{array}
[c]{cl}%
\hat{\xi}_{\tau}(r)= & X^{\tau,\bar{X}^{t,x;\bar{u}}(s);\bar{u}}(r)-\bar
{X}^{t,x;\bar{u}}(r),\ \hat{\eta}_{\tau}(r)=  Y^{\tau,\bar{X}^{t,x;\bar{u}}(s);\bar{u}}(r)-\bar
{Y}^{t,x;\bar{u}}(r),\\
\hat{\zeta}_{\tau}(r)= & Z^{\tau,\bar{X}^{t,x;\bar{u}}(s);\bar{u}}(r)-\bar
{Z}^{t,x;\bar{u}}(r),\  \hat{\Theta}_{\tau}(r)=  (\hat{\xi}_{\tau}(r),\hat{\eta}_{\tau}(r),\hat
{\zeta}_{\tau}(r)).
\end{array}
\]
Then, by Theorem 2.2 in \cite{Hu-JX}, we have that for each $\beta\in
\lbrack2,8]$
\begin{equation}
\begin{array}
[c]{ll}
\mathbb{E}\left[  \left.  \sup\limits_{r\in\lbrack\tau,T]}\left(  |\hat{\xi
}_{\tau}(r)|^{\beta}+|\hat{\eta}_{\tau}(r)|^{\beta}\right)  +\left(
\int_{\tau}^{T}|\hat{\zeta}_{\tau}(r)|^{2}dr\right)  ^{\frac{\beta}{2}%
}\right\vert \mathcal{F}_{\tau}^{t}\right]  \\
\leq C\left\vert \bar{X}%
^{t,x;\bar{u}}(\tau)-\bar{X}^{t,x;\bar{u}}(s)\right\vert ^{\beta},\  P \text{-}a.s.. \label{new-eq-123}
\end{array}
\end{equation}
Note that%
\[
\bar{X}^{t,x;\bar{u}}(\tau)-\bar{X}^{t,x;\bar{u}}(s)=\int_{s}^{\tau
}b(r)dr+\int_{s}^{\tau}\sigma(r)dB(r).
\]
Taking conditional expectation $\mathbb{E}\left[  \cdot|\mathcal{F}_{s}%
^{t}\right]  $ on both sides of (\ref{new-eq-123}), we obtain for a.e. $s\in\lbrack t,T]$,
\begin{equation}
\mathbb{E}\left[  \left.  \sup\limits_{r\in\lbrack\tau,T]}\left(  |\hat{\xi
}_{\tau}(r)|^{\beta}+|\hat{\eta}_{\tau}(r)|^{\beta}\right)  +\left(
\int_{\tau}^{T}|\hat{\zeta}_{\tau}(r)|^{2}dr\right)  ^{\frac{\beta}{2}%
}\right\vert \mathcal{F}_{s}^{t}\right]  \leq O ( \left\vert \tau-s\right\vert
^{\frac{\beta}{2}}),\;P\text{-}a.s.,\label{est-xi-hat-tau}%
\end{equation}
as $\tau \downarrow s $. We rewrite $\hat{\xi}%
_{\tau}(\cdot),$ $\hat{\eta}_{\tau}(\cdot)$ and $\hat{\zeta}_{\tau}(\cdot)$
as
\begin{equation}
\left\{
\begin{array}
[c]{rl}%
d\hat{\xi}_{\tau}(r)= & \left[  \hat{\Theta}_{\tau}(r)Db(r)+\varepsilon
_{\tau1}(r)\right]  dr+\left[  \hat{\Theta}_{\tau}(r)D\sigma(r)+\varepsilon
_{\tau2}(r)\right]  dB(r),\\
\hat{\xi}_{\tau}(\tau)= & -\int_{s}^{\tau}b(r)dr-\int_{s}^{\tau}%
\sigma(r)dB(r),\\
d\hat{\eta}_{\tau}(r)= & -\left[  \hat{\Theta}_{\tau}(r)Dg(r)+\varepsilon
_{\tau3}(r)\right]  dr+\hat{\zeta}_{\tau}(r)dB(r),\text{ }r\in\lbrack
\tau,T],\\
\hat{\eta}_{\tau}(T)= & \phi_{x}(\bar{X}^{t,x;\bar{u}}(T))\hat{\xi}_{\tau
}(T)+\varepsilon_{\tau4}(T),
\end{array}
\right.
\end{equation}
where%
\[%
\begin{array}
[c]{ll}%
\varepsilon_{\tau1}(r)= \left(  \tilde{b}_{x}(r)-b_{x}%
(r)\right)  \hat{\xi}_{\tau}(r)+\left(  \tilde{b}_{y}
(r)-b_{y}(r)\right)  \hat{\eta}_{\tau}(r)+\left(  \tilde{b}_{z}
(r)-b_{z}(r)\right)  \hat{\zeta}_{\tau}(r),\\
\varepsilon_{\tau2}(r)= \left(  \tilde{\sigma}_{x}
(r)-\sigma_{x}(r)\right)  \hat{\xi}_{\tau}(r)+\left(  \tilde{\sigma}%
_{y}(r)-\sigma_{y}(r)\right)  \hat{\eta}_{\tau}(r)+\left(
\tilde{\sigma}_{z}(r)-\sigma_{z}(r)\right)  \hat{\zeta}_{\tau
}(r),\\
\varepsilon_{\tau3}(r)= \left(  \tilde{g}_{x}(r)-g_{x}(r)\right)
\hat{\xi}_{\tau}(r)+\left(  \tilde{g}_{y}(r)-g_{y}(r)\right)
\hat{\eta}_{\tau}(r)+\left(  \tilde{g}_{z}(r)-g_{z}(r)\right)
\hat{\zeta}_{\tau}(r),\\
\varepsilon_{\tau4}(T)=[\tilde{\phi}_{x}(\bar{X}^{t,x;\bar{u}%
}(T))-\phi_{x}(\bar{X}^{t,x;\bar{u}}(T))]\hat{\xi}_{\tau}(T),\\
\tilde{\psi}_{\kappa}(r)=\int_{0}^{1}  \psi_{\kappa
}(r,\bar{\Theta}^{t,x;\bar{u}}(r)+\lambda\hat{\Theta}(r),\bar{u}%
(r)) d\lambda \
\text{ for }\psi=b,\sigma,g,\phi\text{
and }\kappa=x,y,z.
\end{array}
\]
Similar to the proof in Theorem \ref{th-visco-x}, we obtain%
\[
Y^{\tau,X^{t,x;\bar{u}}(s);\bar{u}}(\tau)-\bar{Y}^{t,x;\bar{u}}(\tau
)=p(\tau)\hat{\xi}_{\tau}(\tau)+\frac{1}{2}P(\tau)\hat{\xi}_{\tau}(\tau
)^{2}+o(| \hat{\xi}_{\tau}(\tau)| ^{2}),\;P\text{-}a.s.,
\]
which implies   for a.e. $s\in\lbrack t,T]$,
\[
\mathbb{E}\left[  \left.  Y^{\tau,X^{t,x;\bar{u}}(s);\bar{u}}(\tau)-\bar
{Y}^{t,x;\bar{u}}(\tau)\right\vert \mathcal{F}_{s}^{t}\right]  =\mathbb{E}%
\left[  p(\tau)\hat{\xi}_{\tau}(\tau)+\frac{1}{2}P(\tau)\hat{\xi}_{\tau}%
(\tau)^{2}|\mathcal{F}_{s}^{t}\right]  +o\left(  \left\vert \tau-s\right\vert
\right)  ,\;
\]
 $P$\text{-}a.s. as $\tau\downarrow s$.

\textbf{Step 2: }Completion of the proof.

By Proposition 3.5 in \cite{Hu-JX-DPP}, we get%
\begin{equation}
W(\tau,X^{t,x;\bar{u}}(s))\leq\mathbb{E}\left[  Y^{\tau,X^{t,x;\bar{u}%
}(s);\bar{u}}(\tau)|\mathcal{F}_{s}^{t}\right]  ,\;P\text{-}%
a.s..\label{W-less-Y}%
\end{equation}
Similar to Theorem \ref{th-visco-x}, we can find a subset $\Omega_{0}\subseteq\Omega$ with $P(\Omega_{0})=1$
such that for any $\omega_{0}\in\Omega_{0}$,
\[
\left\{
\begin{array}
[c]{l}%
W(s,\bar{X}^{t,x;\bar{u}}(s,\omega_{0}))=\bar{Y}^{t,x;\bar{u}}(s,\omega
_{0}),\text{ (\ref{est-xi-hat-tau}), (\ref{W-less-Y}) are satisfied for any
rational } \\
 \tau>s\text{, }
(\Omega,\mathcal{F},P(\cdot|\mathcal{F}_{s}^{t})\left(  \omega_{0}\right)
,B(\cdot)-B(s);u(\cdot))|_{[s,T]}\in\mathcal{U}^{w}[s,T],\ \text{and
}\\
\underset{s\leq r\leq T}{\sup}\left[  \left\vert p\left(  r,\omega
_{0}\right)  \right\vert +\left\vert P(r,\omega_{0})\right\vert \right]
<\infty.
\end{array}
\right.
\]
The first relation of the above is a directly application of the DPP (see Theorem 3.6 in
\cite{Hu-JX-DPP}). Let $\omega_{0}\in\Omega_{0}$ be fixed. Then, for any
rational number $\tau>s$ and for a.e. $s\in\lbrack t,T)$
\begin{equation}%
\begin{array}
[c]{rl}%
&W(\tau,\bar{X}^{t,x;\bar{u}}(s,\omega_{0}))-W(s,\bar{X}^{t,x;\bar{u}}%
(s,\omega_{0}))\\
& \leq \mathbb{E}\left[  Y^{\tau,X^{t,x;\bar{u}}(s);\bar{u}%
}(\tau)-\bar{Y}^{t,x;\bar{u}}(s)|\mathcal{F}_{s}^{t}\right]  (\omega_{0})\\
&=  \mathbb{E}\left[  Y^{\tau,X^{t,x;\bar{u}}(s);\bar{u}}(\tau)-\bar
{Y}^{t,x;\bar{u}}(\tau)+\bar{Y}^{t,x;\bar{u}}(\tau)-\bar{Y}^{t,x;\bar{u}%
}(s)|\mathcal{F}_{s}^{t}\right]  (\omega_{0})\\
&= \mathbb{E}\left[  p(\tau)\hat{\xi}_{\tau}(\tau)+\frac{1}{2}P(\tau)\hat
{\xi}_{\tau}(\tau)^{2}-\int_{s}^{\tau}g(r)dr|\mathcal{F}_{s}^{t}\right]
(\omega_{0})+o(\left\vert \tau-s\right\vert ),
\end{array}
\label{eq-new-121}%
\end{equation}
as $\tau\downarrow s$. Next we estimate the terms
on the right hand side of (\ref{eq-new-121}).%
\begin{equation}%
\begin{array}
[c]{cl}%
\mathbb{E}\left[  p(\tau)\hat{\xi}_{\tau}(\tau)|\mathcal{F}_{s}^{t}\right]
(\omega_{0}) & =\mathbb{E}\left[  p(s)\hat{\xi}_{\tau}(\tau)+(p(\tau)-p\left(
s\right)  )\hat{\xi}_{\tau}(\tau)|\mathcal{F}_{s}^{t}\right]  (\omega_{0})\\
& =\mathbb{E}\left[  -p(s)\int_{s}^{\tau}b(r)dr-\int_{s}^{\tau}q(r)\sigma
(r)dr|\mathcal{F}_{s}^{t}\right]  (\omega_{0})+o(\left\vert \tau-s\right\vert
),
\end{array}
\label{eq-new-122}%
\end{equation}
where the last equality is due to the It\^{o}'s formula for $(p(\tau)-p\left(
s\right)  )\hat{\xi}_{\tau}(\tau)$. Similarly,
\begin{equation}%
\begin{array}
[c]{cc}%
\mathbb{E}\left[  \frac{1}{2}P(\tau)\hat{\xi}_{\tau}(\tau)^{2}|\mathcal{F}%
_{s}^{t}\right]  (\omega_{0}) & =\mathbb{E}\left[  \frac{1}{2}P(s)\int%
_{s}^{\tau}\sigma(r)^{2}dr|\mathcal{F}_{s}^{t}\right]  (\omega_{0}%
)+o(\left\vert \tau-s\right\vert ).
\end{array}
\label{eq-new-123}%
\end{equation}
Thus, by (\ref{eq-new-121})-(\ref{eq-new-123}) and the continuity of $W$, we
obtain%
\[%
\begin{array}
[c]{l}%
W(\tau,\bar{X}^{t,x;\bar{u}}(s))-W(s,\bar{X}^{t,x;\bar{u}}(s))\\
\leq\mathbb{E}\left[  -p(s)\int_{s}^{\tau}b(r)dr-\int_{s}^{\tau}%
q(r)\sigma(r)dr-\int_{s}^{\tau}g(r)dr+\frac{1}{2}P(s)\int_{s}^{\tau}%
\sigma(r)^{2}dr|\mathcal{F}_{s}^{t}\right] \\
\ \  +o(\left\vert \tau-s\right\vert
)\\
=(\tau-s)\mathcal{H}_{1}(s,\bar{X}^{t,x;\bar{u}}(s),\bar{Y}^{t,x;\bar{u}%
}(s),\bar{Z}^{t,x;\bar{u}}(s))+o(\left\vert \tau-s\right\vert ),
\end{array}
\]
which implies
\[
\lbrack\mathcal{H}_{1}(s,\bar{X}^{t,x;\bar{u}}(s),\bar{Y}^{t,x;\bar{u}%
}(s),\bar{Z}^{t,x;\bar{u}}(s)),\infty)\subseteq D_{t+}^{1,+}W(s,X^{t,x;\bar
{u}}(s))
\]
by the definition of super-jets. For any $\hat{q}\in D_{t+}^{1,-}W(s,\bar
{X}^{t,x;\bar{u}}(s))$, by definition of sub-jets, we have%
\[%
\begin{array}
[c]{rl}%
0&\leq\underset{\tau\downarrow s}{\liminf}\left\{  \frac{V(\tau,\bar
{X}^{t,x;\bar{u}}(s))-V(s,\bar{X}^{t,x;\bar{u}}(s))-\hat{q}(\tau-s)}{\tau
-s}\right\}  \\
&\leq\underset{\tau\downarrow s}{\lim\inf}\left\{  \mathcal{H}_{1}(s,\bar
{X}^{t,x;\bar{u}}(s),\bar{Y}^{t,x;\bar{u}}(s),\bar{Z}^{t,x;\bar{u}}%
(s))-\hat{q}\right\}  .
\end{array}
\]
Thus
\[
\hat{q}\leq\mathcal{H}_{1}(s,\bar{X}^{t,x;\bar{u}}(s),\bar{Y}^{t,x;\bar{u}%
}(s),\bar{Z}^{t,x;\bar{u}}(s)),\text{ }\forall s\in\lbrack t,T),P\text{-}a.s..
\]
This completes the proof.
\end{proof}

\section{Special cases}

In this section, we study three special cases. In the first case, the value
function $W$ is supposed to be smooth. In the second case, the diffusion term
$\sigma$ of the forward stochastic differential equation in (\ref{state-eq})
is linear in $z$. Finally, we study the case in which the control
domain is convex and compact.

\subsection{The smooth case}

In this subsection, we assume that the value function $W$ is smooth and obtain
the relationship between the derivatives of $W$ and the adjoint processes.
Note that the HJB equation includes an algebra equation (\ref{eq-hjb}). It is
worth pointing out that we discover two novel connections: (i) the relation between the derivatives of $V(\cdot)$ and
the terms $K_{1}(\cdot)$, $K_{2}(\cdot)$ in the adjoint equations; (ii) the relation between the algebra equation $\Delta(\cdot)$ for the MP and the algebra equation $V(\cdot)$ for the HJB equation.
We first give the following stochastic verification theorem.

\begin{theorem}
Suppose Assumption \ref{assum-1} holds. Let
$w(t,x)$ belong to  $ C_{b}^{1,2}([0,T]\times \mathbb{R})$ and be a solution to the HJB
equation (\ref{eq-hjb}). If $||\sigma||_{\infty}<\infty$ and $||w_{x}%
||_{\infty}||\sigma_{z}||_{\infty}<1$, then%
\[
w(t,x)\leq J(t,x;u(\cdot)),\text{ }\forall u(\cdot)\in\mathcal{U}%
^{w}[t,T],(t,x)\in\lbrack0,T]\times\mathbb{R}.
\]
Furthermore, if $\bar{u}(\cdot)\in\mathcal{U}^{w}[t,T]$ such that%
\begin{equation*}
\begin{array}[c]{l}
G(s,X^{t,x;\bar{u}}(s),w(s,X^{t,x;\bar{u}}(s)),w_{x}(s,X^{t,x;\bar{u}%
}(s),w_{xx}(s,X^{t,x;\bar{u}}(s)),\bar{u}(s))\\
+w_{s}(s,X^{t,x;\bar{u}}(s))=0,
\end{array}
\end{equation*}
where $(X^{t,x;\bar{u}}(\cdot),Y^{t,x;\bar{u}}(\cdot),Z^{t,x;\bar{u}}(\cdot))$
is the solution to FBSDE (\ref{state-eq}) corresponding to $\bar{u}(\cdot)$,
%and $V(t,x,v,w_{x}(t,x),u)=w_{x}(t,x)\sigma(s,x,w(s,x),V(t,x,v,w_{x}(t,x),u),u),$ for any $ (s,x)\in\lbrack
%t,T]\times\mathbb{R},\ u\in U$
then $\bar{u}(\cdot)$ is an optimal control.
\end{theorem}

\begin{proof}
The proof is same to Theorem 4.12 in \cite{Hu-JX-DPP}, thus we omit it.
\end{proof}

Now we study the relationship between the derivatives of the value function
$W$ and the adjoint processes.

\begin{theorem}
\label{new-th-11}Let Assumptions \ref{assum-1}, \ref{assum-2} and
\ref{assum-3} hold. Suppose that $\bar{u}(\cdot)\in\mathcal{U}^{w}[t,T]$ is an
optimal control, and $(\bar{X}^{t,x;\bar{u}}(\cdot),\bar{Y}^{t,x;\bar{u}%
}(\cdot),\bar{Z}^{t,x;\bar{u}}(\cdot))$ is the corresponding optimal state.
Let $(p(\cdot),q(\cdot))$ be the solution to (\ref{eq-p}). If the value
function $W(\cdot,\cdot)\in C^{1,2}([t,T]\times\mathbb{R})$, then for each $s\in\lbrack t,T]$
\[
\bar{Y}^{t,x;\bar{u}}(s)=W(s,\bar{X}^{t,x;\bar{u}}(s)),\]  \[\bar
{Z}^{t,x;\bar{u}}(s)=V(s,\bar{X}^{t,x;\bar{u}}(s),W(s,\bar{X}^{t,x;\bar{u}}(s)),W_x(s,\bar{X}^{t,x;\bar{u}}(s)),\bar{u}(s))
\]
and
\[%
\begin{array}
[c]{rl}%
&-W_{s}(s,\bar{X}^{t,x;\bar{u}}(s))\\
& =G(s,\bar{X}^{t,x;\bar{u}}(s),W(s,\bar
{X}^{t,x;\bar{u}}(s)),W_{x}(s,\bar{X}^{t,x;\bar{u}}(s)),W_{xx}(s,\bar
{X}^{t,x;\bar{u}}(s)),\bar{u}(s))\\
& =\min\limits_{u\in U}G\left(  s,\bar{X}^{t,x;\bar{u}}(s),W(s,\bar
{X}^{t,x;\bar{u}}(s)),W_{x}(s,\bar{X}^{t,x;\bar{u}}(s)),W_{xx}(s,\bar
{X}^{t,x;\bar{u}}(s)),u\right).
\end{array}
\]
Moreover, if $W(\cdot,\cdot)\in C^{1,3}([t,T]\times\mathbb{R})$ and $W_{sx}(\cdot,\cdot),$
$W_{sxx}(\cdot,\cdot)$ are continuous, then, for $s\in\lbrack t,T]$,%
\[%
\begin{array}
[c]{rl}%
p(s)= & W_{x}(s,\bar{X}^{t,x;\bar{u}}(s)),\\
q(s)= & W_{xx}(s,\bar{X}^{t,x;\bar{u}}(s))\sigma(s,\bar{X}^{t,x;\bar{u}%
}(s),\bar{Y}^{t,x;\bar{u}}(s),\bar{Z}^{t,x;\bar{u}}(s),\bar{u}(s)).
\end{array}
\]
Furthermore, if $W(\cdot,\cdot)\in C^{1,4}([t,T]\times\mathbb{R})$ and
$W_{sxx}(\cdot,\cdot)$ is continuous, then
\[
P(s)\geq W_{xx}(s,\bar{X}^{t,x;\bar{u}}(s))\text{, }s\in\lbrack t,T],
\]
where $(P(\cdot),Q(\cdot))$ satisfies (\ref{eq-P}).
\end{theorem}

\begin{proof}
By the DPP (see Theorem 3.6 in \cite{Hu-JX-DPP}), we get $\bar{Y}^{t,x;\bar{u}}(s)=W(s,\bar
{X}^{t,x;\bar{u}}(s))$ $s\in\lbrack t,T]$. Applying It\^{o}'s formula to
$W(s,\bar{X}^{t,x;\bar{u}}(s))$, we can get%
\begin{equation}%
\begin{array}
[c]{l}%
\bar{Y}^{t,x;\bar{u}}(s)=W(s,\bar{X}^{t,x;\bar{u}}(s)),\\ \bar
{Z}^{t,x;\bar{u}}(s)=V(s,\bar{X}^{t,x;\bar{u}}(s),W(s,\bar{X}^{t,x;\bar{u}}(s)),W_x(s,\bar{X}^{t,x;\bar{u}}(s)),\bar{u}(s)),\\
G(s,\bar{X}^{t,x;\bar{u}}(s),W(s,\bar
{X}^{t,x;\bar{u}}(s)),W_{x}(s,\bar{X}^{t,x;\bar{u}}(s)),W_{xx}(s,\bar
{X}^{t,x;\bar{u}}(s)),\bar{u}(s))\\
+W_{s}(s,\bar{X}^{t,x;\bar{u}}(s))=0.
\end{array}
\label{new-asd-1323}%
\end{equation}
Since $W$ satisfies the HJB equation (\ref{eq-hjb}), we obtain that, for each
$u\in U$,%
\begin{equation*}
\begin{array}[c]{ll}
G\left(  s,\bar{X}^{t,x;\bar{u}}%
(s),W(s,\bar{X}^{t,x;\bar{u}}(s)),W_{x}(s,\bar{X}^{t,x;\bar{u}}(s)),W_{xx}%
(s,\bar{X}^{t,x;\bar{u}}(s)),u\right)\\
+ W_{s}(s,\bar{X}^{t,x;\bar{u}}(s)) \geq0.
\end{array}
\end{equation*}
Thus we deduce%
\begin{equation}%
\begin{array}
[c]{l}%
G(s,\bar{X}^{t,x;\bar{u}}(s),W(s,\bar{X}^{t,x;\bar{u}}(s)),W_{x}(s,\bar
{X}^{t,x;\bar{u}}(s)),W_{xx}(s,\bar{X}^{t,x;\bar{u}}(s)),\bar{u}(s))\\
=\min\limits_{u\in U}G\left(  s,\bar{X}^{t,x;\bar{u}}(s),W(s,\bar{X}%
^{t,x;\bar{u}}(s)),W_{x}(s,\bar{X}^{t,x;\bar{u}}(s)),W_{xx}(s,\bar
{X}^{t,x;\bar{u}}(s)),u\right)  .
\end{array}
\label{neq-dfds-11}%
\end{equation}

If $W(\cdot,\cdot)\in C^{1,3}([t,T]\times\mathbb{R})$ and $W_{sx}(\cdot
,\cdot)$ is continuous, then, by applying It\^{o}'s formula to $W_{x}%
(s,\bar{X}^{t,x;\bar{u}}(s))$, we get%
\begin{equation}%
\begin{array}
[c]{ll}%
&dW_{x}(s,\bar{X}^{t,x;\bar{u}}(s))\\
&=  \left\{  W_{sx}(s,\bar{X}^{t,x;\bar{u}%
}(s))+W_{xx}(s,\bar{X}^{t,x;\bar{u}}(s))b(s)+\frac{1}{2}W_{xxx}(s,\bar
{X}^{t,x;\bar{u}}(s))(\sigma(s))^{2}\right\}  ds\\
&\ \  +W_{xx}(s,\bar{X}^{t,x;\bar{u}}(s))\sigma(s)dB(s).
\end{array}
\label{eq-W-x}%
\end{equation}
Note that $W$ satisfies the HJB equation (\ref{eq-hjb}). Then we obtain
\begin{equation}
W_{s}(s,x)+G(s,x,W(s,x),W_{x}(s,x),W_{xx}(s,x),\bar{u}(s))\geq
0.\label{new-asd-1322}%
\end{equation}
Combining (\ref{new-asd-1323})\ and (\ref{new-asd-1322}), we conclude that the
function $$W_{s}(s,\cdot)+G(s,\cdot,W(s,\cdot),W_{x}(s,\cdot),W_{xx}%
(s,\cdot),\bar{u}(s))$$ achieves its minimum at $x=\bar{X}^{t,x;\bar{u}}(s)$.
Thus%
\begin{equation}
\left.  \frac{\partial}{\partial x}(W_{s}(s,x)+G(s,x,W(s,x),W_{x}(s,x),W_{xx}(s,x),\bar
{u}(s)))\right\vert _{x=\bar{X}^{t,x;\bar{u}}(s)}=0.\label{eq-new-1211}%
\end{equation}
By the implicit function theorem, we deduce
\begin{equation}%
\begin{array}
[c]{l}%
\left. \frac{\partial V}{\partial x}(s,x,W(s,x),W_x(s,x),\bar{u}(s))\right\vert _{x=\bar{X}^{t,x;\bar{u}}(s)}\\
=\left(  1-W_{x}(s,\bar{X}^{t,x;\bar{u}}(s))\sigma_{z}(s)\right)  ^{-1}%
[W_{xx}(s,\bar{X}^{t,x;\bar{u}}(s))\sigma(s)+W_{x}(s,\bar{X}^{t,x;\bar{u}%
}(s))\sigma_{x}(s)\\
\ \ +\sigma_{y}(s)(W_{x}(s,\bar{X}^{t,x;\bar{u}}(s)))^{2}].
\end{array}
\label{eq-V_x}%
\end{equation}
Thus, we can easily get
\begin{equation}%
\begin{array}
[c]{l}%
\left.  \frac{\partial}{\partial x}(W_{s}(s,x)+G(s,x,W(s,x),W_{x}(s,x),W_{xx}(s,x),\bar
{u}(s)))\right\vert _{x=\bar{X}^{t,x;\bar{u}}(s)}\\
=W_{sx}(s,\bar{X}^{t,x;\bar{u}}(s))+W_{xx}(s,\bar{X}^{t,x;\bar{u}%
}(s))b(s)+\frac{1}{2}W_{xxx}(s,\bar{X}^{t,x;\bar{u}}(s))\sigma(s)^{2}\\
\text{ }+W_{x}(s,\bar{X}^{t,x;\bar{u}}(s))[b_{x}(s)+b_{y}(s)W_{x}(s,\bar
{X}^{t,x;\bar{u}}(s))+b_{z}(s)V_x(s)]\\
\text{ }+W_{xx}(s,\bar{X}^{t,x;\bar{u}}(s))\sigma(s)\left[\sigma_{x}(s)+\sigma
_{y}(s)W_{x}(s,\bar{X}^{t,x;\bar{u}}(s))+\sigma_{z}(s)V_x(s)\right]\\
\text{ }+g_{x}(s)+g_{y}(s)W_{x}(s,\bar{X}^{t,x;\bar{u}}(s))+g_{z}(s)V_x(s),
\end{array}
\label{deri-hjb}%
\end{equation}
where $$V_x(s)=\left. \frac {\partial V}{\partial x}(s,x,W(s,x),W_x(s,x),\bar{u}(s))\right\vert _{x=\bar{X}^{t,x;\bar{u}}(s)}.$$
Combining (\ref{eq-W-x}), (\ref{eq-new-1211}) and (\ref{deri-hjb}), it is easy to
check that $$\left(  W_{x}(s,\bar{X}^{t,x;\bar{u}}(s)),W_{xx}(s,\bar
{X}^{t,x;\bar{u}}(s))\sigma(s)\right)$$ satisfies the adjoint equation
(\ref{eq-p}), which implies%
\[
p(s)=W_{x}(s,\bar{X}^{t,x;\bar{u}}(s)),\text{ }q(s)=W_{xx}(s,\bar{X}%
^{t,x;\bar{u}}(s))\sigma(s).
\]

If $W(\cdot,\cdot)\in C^{1,4}([t,T]\times\mathbb{R})$ and $W_{sxx}(\cdot
,\cdot)$ is continuous, then, applying It\^{o}'s formula to $W_{xx}(s,\bar
{X}^{t,x;\bar{u}}(s))$, we obtain%
\begin{equation}%
\begin{array}
[c]{ll}%
&dW_{xx}(s,\bar{X}^{t,x;\bar{u}}(s))\\
&= \left\{  W_{sxx}(s,\bar{X}^{t,x;\bar
{u}}(s))+W_{xxx}(s,\bar{X}^{t,x;\bar{u}}(s))b(s)\right.\\
&\left.+\frac{1}{2}W_{xxxx}(s,\bar
{X}^{t,x;\bar{u}}(s))(\sigma(s))^{2}\right\}  ds +W_{xxx}(s,\bar{X}^{t,x;\bar{u}}(s))\sigma(s)dB(s).
\end{array}
\label{new-erds11}%
\end{equation}
Since the function $W_{s}(s,\cdot)+G(s,\cdot,W(s,\cdot),W_{x}(s,\cdot
),W_{xx}(s,\cdot),\bar{u}(s))$ achieves its minimum at $x=\bar{X}^{t,x;\bar
{u}}(s)$, we have%
\begin{equation}
\left.  \frac{\partial^{2}}{\partial x^{2}}(W_{s}(s,x)+G(s,x,W(s,x),W_{x}(s,x),W_{xx}%
(s,x),\bar{u}(s)))\right\vert _{x=\bar{X}^{t,x;\bar{u}}(s)}\geq
0.\label{new-erds12}%
\end{equation}
Set $\tilde{P}(s)=W_{xx}(s,\bar{X}^{t,x;\bar{u}}(s))$ and $\tilde
{Q}(s)=W_{xxx}(s,\bar{X}^{t,x;\bar{u}}(s))\sigma(s)$ for $s\in\lbrack t,T]$.
In order to prove $P(s)\geq\tilde{P}(s)$, by comparison theorem of BSDE for
equations (\ref{eq-P}) and (\ref{new-erds11}), we only need to check
\begin{equation}%
\begin{array}
[c]{l}%
\tilde{P}(s)\left[  (D\sigma(s)^{\intercal}(1,p(s),K_{1}(s)))^{2}+2Db(s)^{\intercal}%
(1,p(s),K_{1}(s))^{\intercal}+H_{y}(s)\right]  \\
+2\tilde{Q}(s)D\sigma(s)^{\intercal}(1,p(s),K_{1}(s))^{\intercal}+(1,p(s),K_{1}%
(s))D^{2}H(s)(1,p(s),K_{1}(s))^{\intercal}\\
+H_{z}(s)\tilde{K}_{2}(s)+W_{sxx}(s,\bar{X}^{t,x;\bar{u}}(s))+W_{xxx}(s,\bar{X}^{t,x;\bar{u}%
}(s))b(s)\\
+\frac{1}{2}W_{xxxx}(s,\bar{X}^{t,x;\bar{u}}(s))(\sigma(s))^{2}\geq0,
\end{array}
\label{new-erda11}%
\end{equation}
where%
\begin{equation}%
\begin{array}
[c]{ll}%
&\tilde{K}_{2}(s)\\
&= (1-p(s)\sigma_{z}(s))^{-1}\left\{  p(s)\sigma
_{y}(s)+2\left[  \sigma_{x}(s)+\sigma_{y}(s)p(s)+\sigma_{z}(s)K_{1}(s)\right]
\right\}  \tilde{P}(s)\\
&\ \  +(1-p(s)\sigma_{z}(s))^{-1}\left\{  \tilde{Q}(s)+p(s)(1,p(s),K_{1}%
(s))D^{2}\sigma(s)(1,p(s),K_{1}(s))^{\intercal}\right\}  .
\end{array}
\label{K2-til}%
\end{equation}
By (\ref{new-erds12}), one can verify that the inequality (\ref{new-erda11}) holds.

From the proof in the above theorem, we can obtain the following corollary.
\end{proof}

\begin{corollary}
Under the same assumptions as in Theorem \ref{new-th-11}, we have the
following relation:%
\[%
\begin{array}
[c]{rl}%
\left. \frac{\partial V}{\partial x}(s,x,W(s,x),W_x(s,x),\bar{u}(s))\right\vert _{x=\bar{X}^{t,x;\bar{u}}(s)} & =K_{1}(s),\\
\left. \frac{\partial^2V}{\partial x^2}(s,x,W(s,x),W_x(s,x),\bar{u}(s))\right\vert _{x=\bar{X}^{t,x;\bar{u}}(s)} & =\tilde{K}_{2}(s),
\end{array}
\]
where $\tilde{K}_{2}(s)$ is defined in (\ref{K2-til}).
\end{corollary}

\begin{remark}
It is worth to pointing out that $\tilde{K}_{2}(\cdot)$ and $K_{2}(\cdot)$ are
closely related. If we replace $P(\cdot)$ (resp. $Q(\cdot)$) by $W_{xx}%
(\cdot,\bar{X}^{t,x;\bar{u}}(\cdot))$ (resp. $W_{xxx}(\cdot,\bar{X}%
^{t,x;\bar{u}}(\cdot))\sigma(\cdot)$) in $K_{2}(\cdot)$, then we have
$\tilde{K}_{2}(\cdot)$.
\end{remark}

If the value function is smooth enough, we can use the DPP to derive the MP in
the following theorem.

\begin{theorem}
Let Assumptions \ref{assum-1}, \ref{assum-2} and \ref{assum-3} hold. Suppose
that $\bar{u}(\cdot)\in\mathcal{U}^{w}[t,T]$ is an optimal control, and
$(\bar{X}^{t,x;\bar{u}}(\cdot),\bar{Y}^{t,x;\bar{u}}(\cdot),\bar{Z}%
^{t,x;\bar{u}}(\cdot))$ is the corresponding optimal state. Let $(p(\cdot
),q(\cdot))$ and $(P(\cdot),Q(\cdot))$ be the solutions to (\ref{eq-p}) and
(\ref{eq-P}) respectively. If $W(\cdot,\cdot)\in C^{1,4}([t,T]\times
\mathbb{R})$ and $W_{sx}(\cdot,\cdot)$, $W_{sxx}(\cdot,\cdot)$ are continuous, then
\begin{equation}
\begin{array}
[c]{l}%
\mathcal{H}(s,\bar{X}^{t,x;\bar{u}}(s),\bar{Y}^{t,x;\bar{u}}(s),\bar
{Z}^{t,x;\bar{u}}(s),u,p(s),q(s),P(s))\\
\geq\mathcal{H}(s,\bar{X}^{t,x;\bar{u}}(s),\bar{Y}^{t,x;\bar{u}}(s),\bar
{Z}^{t,x;\bar{u}}(s),\bar{u}(s),p(s),q(s),P(s)),\ \ \ \forall u\in
U\ a.e.,\ a.s..
\end{array}
\label{new-eqed122}
\end{equation}
\end{theorem}

\begin{proof}
By (\ref{neq-dfds-11}) in Theorem \ref{new-th-11}, we have $\forall u\in U\ a.e.,\ a.s.$
\begin{equation}%
\begin{array}
[c]{l}%
G(s,\bar{X}^{t,x;\bar{u}}(s),W(s,\bar{X}^{t,x;\bar{u}}(s)),W_{x}(s,\bar
{X}^{t,x;\bar{u}}(s)),W_{xx}(s,\bar{X}^{t,x;\bar{u}}(s)),\bar{u}(s))\\
\leq G(s,\bar{X}^{t,x;\bar{u}}(s),W(s,\bar{X}^{t,x;\bar{u}}(s)),W_{x}%
(s,\bar{X}^{t,x;\bar{u}}(s)),W_{xx}(s,\bar{X}^{t,x;\bar{u}}(s)),u).
\end{array}
\label{esjd-11}%
\end{equation}
Since%
\[%
\begin{array}
[c]{rl}%
\bar{Y}^{t,x;\bar{u}}(s)= & W(s,\bar{X}^{t,x;\bar{u}}(s)),\ \bar{Z}^{t,x;\bar{u}}(s)=  W_{x}(s,\bar{X}^{t,x;\bar{u}}(s))\sigma(s),\\
p(s)= & W_{x}(s,\bar{X}^{t,x;\bar{u}}(s)),\ q(s)=  W_{xx}(s,\bar{X}^{t,x;\bar{u}}(s))\sigma(s),
\end{array}
\]
we can obtain%
\begin{equation}
V(s,\bar{X}^{t,x;\bar{u}}(s),W(s,\bar{X}^{t,x;\bar{u}}(s)),W_x(s,\bar{X}^{t,x;\bar{u}}(s)),u)=\bar{Z}^{t,x;\bar{u}}(s)+\Delta
(s)\label{esjd-12}%
\end{equation}
by the definition of $\Delta(s)$ in equation (\ref{def-delt}). Combining
(\ref{esjd-11}) and (\ref{esjd-12}), we deduce that%
\[%
\begin{array}
[c]{l}%
\mathcal{H}(s,\bar{X}^{t,x;\bar{u}}(s),\bar{Y}^{t,x;\bar{u}}(s),\bar
{Z}^{t,x;\bar{u}}(s),u,p(s),q(s),P(s))\\
\ \ -\mathcal{H}(s,\bar{X}^{t,x;\bar{u}%
}(s),\bar{Y}^{t,x;\bar{u}}(s),\bar{Z}^{t,x;\bar{u}}(s),\bar{u}%
(s),p(s),q(s),P(s))\\
\geq\frac{1}{2}\left(  P(s)-W_{xx}(s,\bar{X}^{t,x;\bar{u}}(s))\right)  \left(
\sigma(s,\bar{X}^{t,x;\bar{u}}(s),\bar{Y}^{t,x;\bar{u}}(s),\bar{Z}%
^{t,x;\bar{u}}(s),u)-\sigma(s)\right)  ^{2}.
\end{array}
\]
Noting that $P(s)\geq W_{xx}(s,\bar{X}^{t,x;\bar{u}}(s))$, then we obtain
(\ref{new-eqed122}).
\end{proof}

\subsection{The case that $\sigma$ is linear in $z$}

In this subsection, we consider the case that  $\sigma(t,x,y,z,u)=\tilde{A}(t)z+\sigma_1(t,x,y,u)$.
Under this case, we do not need the assumption that $q(\cdot)$ is bounded.
\begin{assumption}\label{sig-lin}
$\sigma(t,x,y,x,u)=\tilde{A}(t)z+\sigma_1(t,x,y,u)$,\ $||\tilde{A}(\cdot)||_{\infty}$ \text{is small enough.}
\end{assumption}
\begin{theorem}
\label{th-visco-xx}Suppose Assumptions \ref{assum-1}, \ref{assum-2},
\ref{assum-3} and \ref{sig-lin} hold. Let $\bar{u}(\cdot)$ be optimal for our problem
\eqref{obje-eq}, and let $(p(\cdot),q(\cdot))$ $\in L_{\mathbb{F}}^{\infty}(0,T;\mathbb{R}) \times L_{\mathbb{F}}%
^{2,2}(0,T;\mathbb{R}) $ and $(P(\cdot),Q(\cdot))$ $\in
L_{\mathbb{F}}^{2}(\Omega;C([0,T],\mathbb{R})) \times L_{\mathbb{F}}%
^{2,2}(0,T;\mathbb{R})$ be the solution to equation \eqref{eq-p} and
\eqref{eq-P} respectively. Then
\[
\left\{
\begin{array}
[c]{c}%
\{p(s)\}\times\lbrack P(s),\infty)\subseteq D_{x}^{2,+}W(s,\bar{X}%
^{t,x;\bar{u}}(s)),\\
D_{x}^{2,-}W(s,\bar{X}^{t,x;\bar{u}}(s))\subseteq\{p(s)\}\times(-\infty,P(s)].
\end{array}
\right.
\]

\end{theorem}

\begin{proof}
We use the same notations as in the proof of Theorem \ref{th-visco-x}. In this case
 \begin{equation}
\varepsilon_{2}\left(  r\right) =\left(  \tilde{\sigma}_{x}
(r)-\sigma_{x}(r)\right)  \hat{X}(r)+\left(  \tilde{\sigma}_{y}
(r)-\sigma_{y}(r)\right) \left(p(r) \hat{X}\left(  r\right)+\nu(r)\right).
\end{equation}
It is easy to verify that
 \begin{equation}\label{bound-A}
\begin{array}
[c]{l}%
|A\left(  r\right)  |\leq C\left(  1+\left\vert q\left(  r\right)  \right\vert
\right)  ,\\
|C(r)|\leq C\left(  1+ ||\tilde{A}(  \cdot)  || _{\infty}  \left\vert q\left(  r\right)  \right\vert \right),
\end{array}
 \end{equation}
where $C$ is a positive constant. By Theorem 5.2 in \cite{Hu-JX}, there exists a $\delta>0$ such that for each $\lambda_{1}<\delta$,
\begin{equation}
\begin{array}
[c]{l}%
\mathbb{E}\left[  \left.  \exp\left(  \lambda_{1}\int_{s}^{T}\left\vert
q\left(  r\right)  \right\vert ^{2}dr\right)  \right\vert \mathcal{F}_{s}%
^{t}\right]  \leq C ,\\
 \mathbb{E}\left[  \left.  \sup_{s\leq\alpha\leq
T}\exp\left(  \lambda_{1}\int_{s}^{\alpha}q\left(  r\right)  dB(r)\right)
\right\vert \mathcal{F}_{s}^{t}\right]  \leq C.
\end{array}
\label{boundC}%
\end{equation}
Set, for $r\in\lbrack s,T]$,%
\[
\Gamma_{1}(r)=\exp\left(  \int_{s}^{r}A(\alpha)d\alpha\right)  ,\text{ }%
\Gamma_{2}(r)=\exp\left(  -\frac{1}{2}\int_{s}^{r}|C(\alpha)|^{2}d\alpha
+\int_{s}^{r}C(\alpha)dB(\alpha)\right)  .
\]
When $||\tilde{A}(  \cdot)  || _{\infty}$ is
small enough, by (\ref{bound-A}) and (\ref{boundC}), we can find a large
enough constant $\lambda>0$ such that%
\begin{equation}
\mathbb{E}\left[  \left.  \sup_{r\in\lbrack t,T]}|\Gamma_{1}(r)\Gamma
_{2}(r)|^{\lambda}\right\vert \mathcal{F}_{s}^{t}\right]  \leq
C.\label{eq-edsd-2}%
\end{equation}
Set
\[%
\begin{array}
[c]{l}%
\varphi^{1}\left(  r\right)  =\varphi(r)\Gamma_{1}(r)\Gamma_{2}(r),\\
\nu^{1}\left(  r\right)  =\nu\left(  r\right)  \Gamma_{1}(r)\Gamma
_{2}(r)+C\left(  r\right)  \Gamma_{1}(r)\Gamma_{2}(r)\varphi\left(  r\right)
,
\end{array}
\]
where $(\varphi(\cdot),\nu(\cdot))$ is the solution to BSDE
(\ref{eq-varphi}). We obtain that $\left(  \varphi^{1}\left(
\cdot\right)  ,\nu^{1}\left(  \cdot\right)  \right)  $ satisfies the following
BSDE by applying It\^{o}'s formula to $\varphi(\cdot)\Gamma
_{1}(\cdot)\Gamma_{2}(\cdot)$,
\begin{equation}
\left\{
\begin{array}
[c]{rl}%
d\varphi^{1}\left(  r\right)  = & -\mathrm{II}\left(  r\right)  dr+\nu^{1}\left(
r\right)  dB\left(  r\right)  ,\\
\varphi^{1}\left(  T\right)  = & \varepsilon_{4}\left(  T\right)  \Gamma
_{1}(T)\Gamma_{2}(T),
\end{array}
\right.  \label{def-phi1}%
\end{equation}
where
\begin{equation}
\begin{array}
[c]{ll}%
\mathrm{II}\left(  r\right) =&\Gamma_{1}(r)\Gamma_{2}(r)[p(r)\varepsilon_{1}\left(
r\right)  +q\left(  r\right)  \varepsilon_{2}\left(  r\right)  +\varepsilon
_{3}\left(  r\right) \\
 &+H_{z}\left(  r\right)  \left(  1-p\left(  r\right)
\tilde{A}\left(  r\right)  \right)  ^{-1}p\left(  r\right)  \varepsilon
_{2}\left(  r\right)  ].
\end{array}
\end{equation}
By the estimate of BSDE that, for each $\beta\in\lbrack2,3]$, we have
\begin{equation}%
\begin{array}
[c]{l}%
\mathbb{E}\left[  \left.  |\varepsilon_{4}(T)\Gamma_{1}(T)\Gamma
_{2}(T)|^{\beta}\right\vert \mathcal{F}_{s}^{t}\right]  \\
\leq\left\{  \mathbb{E}\left[  \left.  |\varepsilon_{4}(T)|^{4}\right\vert
\mathcal{F}_{s}^{t}\right]  \right\}  ^{\frac{\beta}{4}}\left\{
\mathbb{E}\left[  \left.  |\Gamma_{1}(T)\Gamma_{2}(T)|^{\frac{4\beta}{4-\beta
}}\right\vert \mathcal{F}_{s}^{t}\right]  \right\}  ^{\frac{4-\beta}{4}}\\
\leq C\left\vert x^{\prime}-\bar{X}^{t,x;\bar{u}%
}(s)\right\vert ^{2\beta},
\end{array}
\end{equation}%
\begin{equation}%
\begin{array}
[c]{l}%
\mathbb{E}\left[  \left.  \left(  \int_{s}^{T}|\mathrm{II}(r)|dr\right)  ^{\beta
}\right\vert \mathcal{F}_{s}^{t}\right]  \\
\leq C\left\{  \mathbb{E}\left[  \left.  \left(  \int_{s}^{T}(|\varepsilon
_{1}(r)|+(1+|q(r)|)|\varepsilon_{2}(r)|+|\varepsilon_{3}(r)|)dr\right)
^{\frac{7}{2}}\right\vert \mathcal{F}_{s}^{t}\right]  \right\}^{\frac
{2\beta}{7}}\\
\ \ \ \ \cdot\left\{  \mathbb{E}\left[  \left.  |\sup_{r\in\lbrack t,T]}%
|\Gamma_{1}(r)\Gamma_{2}(r)|^{\frac{7\beta}{7-2\beta}}\right\vert
\mathcal{F}_{s}^{t}\right]  \right\}  ^{\frac{7-2\beta}{7}}\\
\leq C \left\vert x^{\prime}-\bar{X}^{t,x;\bar{u}%
}(s)\right\vert ^{2\beta},
\end{array}
\end{equation}
then
\begin{equation}%
\begin{array}
[c]{l}%
\mathbb{E}\left[  \left.  \sup\limits_{r\in\lbrack s,T]}|\varphi
^{1}(r)|^{\beta}+\left(  \int_{s}^{T}|\nu^{1}(r)|^{2}dr\right)  ^{\frac{\beta
}{2}}\right\vert \mathcal{F}_{s}^{t}\right]  \\
\leq C\mathbb{E}\left[  \left.  |\varepsilon_{4}(T)\Gamma_{1}(T)\Gamma
_{2}(T)|^{\beta}+\left(  \int_{s}^{T}|\mathrm{II}(r)|dr\right)  ^{\beta}\right\vert
\mathcal{F}_{s}^{t}\right]  \\
\leq \left\vert x^{\prime}-\bar{X}^{t,x;\bar{u}%
}(s)\right\vert ^{2\beta}.
\end{array}
\label{eq-edsd-3}%
\end{equation}
Combining (\ref{eq-edsd-2}) (\ref{def-phi1}) and (\ref{eq-edsd-3}), we obtain
that, for each $\beta\in\lbrack2,\frac{5}{2}]$,%
\begin{equation}
\mathbb{E}\left[  \left.  \sup\limits_{r\in\lbrack s,T]}|\varphi(r)|^{\beta
}+\left(  \int_{s}^{T}|\nu(r)|^{2}dr\right)  ^{\frac{\beta}{2}}\right\vert
\mathcal{F}_{s}^{t}\right]  \leq C \left\vert x^{\prime}-\bar{X}^{t,x;\bar{u}%
}(s)\right\vert ^{2\beta}.\label{eq-edsd-4}%
\end{equation}

Similar to the above analysis, there exists a large enough $\lambda>0$ such that for
\begin{equation}
\mathbb{E}\left[  \left.  \sup\limits_{r\in\lbrack s,T]}|P(r)|^{\lambda}+\left(  \int_{s}^{T}|Q(r)|^{2}dr\right)  ^{\frac{\lambda}{2}}\right\vert
\mathcal{F}_{s}^{t}\right] \leq C.
\end{equation}
Applying It\^{o}'s formula to $\hat{\varphi}(r)\Gamma_{1}(r)\Gamma_{2}(r)$,
where $(\hat{\varphi}(\cdot),\hat{\nu}(\cdot))$ is the equation
(\ref{eq-510}) in Step 4 in the proof of Theorem \ref{th-visco-x}, we get%
\begin{equation}
\hat{\varphi}(s)=\mathbb{E}\left[  \left.  \Gamma_{1}(T)\Gamma_{2}%
(T)\varepsilon_{8}(T)+\int_{s}^{T}\Gamma_{1}(r)\Gamma_{2}(r)I(r)dr\right\vert
\mathcal{F}_{s}^{t}\right]  .\label{eq-edsd-5}%
\end{equation}

By (\ref{eq-edsd-2}) and (\ref{eq-edsd-5}), we deduce that
\begin{equation}
|\hat{\varphi}(s)|\leq C\left\{  \mathbb{E}\left[  \left.  |\varepsilon
_{8}(T)|^{\frac{9}{8}}+\left(  \int_{s}^{T}|\mathrm{I}(r)|dr\right)
^{\frac{9}{8}}\right\vert \mathcal{F}_{s}^{t}\right]  \right\}  ^{\frac{8}{9}%
}.\label{eq-edsd-6}%
\end{equation}
The estimate for $\varepsilon_{8}(T)$ is the same as Theorem 3.1. Since we relax the assumption of $q(\cdot)$ is bounded, the corresponding
$\mathrm{I}(r)$ in (\ref{def-I}) has the following upper bound
\[%
\begin{array}
[c]{rl}%
\left\vert \mathrm{I}\left(  r\right)  \right\vert \leq & C\left[  1+\left(
1+\left\vert q\left(  r\right)  \right\vert \right)  \left\vert P\left(
r\right)  \right\vert +\left\vert Q\left(  r\right)  \right\vert +\left\vert
q\left(  r\right)  \right\vert \right]  \\
& \cdot\left(  \left\vert \varphi\left(  r\right)  \right\vert +\left\vert
\nu\left(  r\right)  \right\vert +\rho\left(  r\right)  \left\vert \hat
{X}\left(  r\right)  \right\vert \right)  \left\vert \hat{X}\left(  r\right)
\right\vert +C(1+|q(r)|^{2})\rho(r)|\hat{X}(r)|^{2}\\
& +C\left(  1+\left\vert q\left(  r\right)  \right\vert +\left\vert P\left(
r\right)  \right\vert \right)  \varphi^{2}\left(  r\right)  +C\left(
1+\left\vert P\left(  r\right)  \right\vert \right)  \left\vert \nu\left(
r\right)  \right\vert ^{2},
\end{array}
\]
where $\rho(r)$ is the same as (\ref{def-rho}).
We estimate the following terms:
\[%
\begin{array}
[c]{l}%
\mathbb{E}\left[  \left.  \left(  \int_{s}^{T}\left\vert q\left(  r\right)
\right\vert P\left(  r\right)  \rho\left(  r\right)  \left\vert \hat{X}\left(
r\right)  \right\vert ^{2}dr\right)  ^{\frac{9}{8}}\right\vert \mathcal{F}%
_{s}^{t}\right]  \\
\leq\mathbb{E}\left[  \left.  \sup\limits_{r\in\lbrack t,T]}\left\vert \hat
{X}\left(  r\right)  \right\vert ^{\frac{9}{4}}\sup\limits_{r\in\lbrack
t,T]}\left\vert P\left(  r\right)  \right\vert ^{\frac{9}{8}}\left(  \int%
_{s}^{T}\left\vert q\left(  r\right)  \right\vert ^{2}dr\right)  ^{\frac
{9}{16}}\left(  \int_{s}^{T}\left\vert \rho\left(  r\right)  \right\vert
^{2}dr\right)  ^{\frac{9}{16}}\right\vert \mathcal{F}_{s}^{t}\right]  \\
\leq\left\{  \mathbb{E}\left[  \left.  \sup\limits_{r\in\lbrack t,T]}%
\left\vert \hat{X}\left(  r\right)  \right\vert ^{8}\right\vert \mathcal{F}%
_{s}^{t}\right]  \right\}  ^{\frac{9}{32}}\\
\text{ \ \ }\cdot\left\{  \mathbb{E}\left[  \left.  \sup\limits_{r\in\lbrack
t,T]}\left\vert P\left(  r\right)  \right\vert ^{\frac{36}{23}}\left(
\int_{s}^{T}\left\vert q\left(  r\right)  \right\vert ^{2}dr\right)
^{\frac{18}{23}}\left(  \int_{s}^{T}\left\vert \rho\left(  r\right)
\right\vert ^{2}dr\right)  ^{\frac{18}{23}}\right\vert \mathcal{F}_{s}%
^{t}\right]  \right\}  ^{\frac{23}{32}}\\
=o\left(  \left\vert x-\bar{X}^{t,x;\bar{u}}\left(  s\right)  \right\vert
^{\frac{9}{4}}\right);
\end{array}
\]%
\[%
\begin{array}
[c]{l}%
\mathbb{E}\left[  \left.  \left(  \int_{s}^{T}\left\vert q\left(  r\right)
\right\vert P\left(  r\right)  \left\vert \nu\left(  r\right)  \right\vert
\left\vert \hat{X}\left(  r\right)  \right\vert dr\right)  ^{\frac{9}{8}%
}\right\vert \mathcal{F}_{s}^{t}\right]  \\
\leq\mathbb{E}\left[  \left.  \sup\limits_{r\in\lbrack t,T]}\left\vert \hat
{X}\left(  r\right)  \right\vert ^{\frac{9}{8}}\sup\limits_{r\in\lbrack
t,T]}\left\vert P\left(  r\right)  \right\vert ^{\frac{9}{8}}\left(  \int%
_{s}^{T}\left\vert q\left(  r\right)  \right\vert ^{2}dr\right)  ^{\frac
{9}{16}}\left(  \int_{s}^{T}\left\vert \nu\left(  r\right)  \right\vert
^{2}dr\right)  ^{\frac{9}{16}}\right\vert \mathcal{F}_{s}^{t}\right]  \\
\leq\left\{  \mathbb{E}\left[  \left.  \sup\limits_{r\in\lbrack t,T]}%
\left\vert \hat{X}\left(  r\right)  \right\vert ^{\frac{45}{22}}%
\sup\limits_{r\in\lbrack t,T]}\left\vert P\left(  r\right)  \right\vert
^{\frac{45}{22}}\left(  \int_{s}^{T}\left\vert q\left(  r\right)  \right\vert
^{2}dr\right)  ^{\frac{45}{44}}\right\vert \mathcal{F}_{s}^{t}\right]
\right\}  ^{\frac{11}{20}}\\
\text{ \ }\cdot\left\{  \mathbb{E}\left[  \left.  \left(  \int_{s}%
^{T}\left\vert \nu\left(  r\right)  \right\vert ^{2}dr\right)  ^{\frac{5}{4}%
}\right\vert \mathcal{F}_{s}^{t}\right]  \right\}  ^{\frac{9}{20}}\\
=o\left(  \left\vert x-\bar{X}^{t,x;\bar{u}}\left(  s\right)  \right\vert
^{\frac{9}{4}}\right);
\end{array}
\]%
\[%
\begin{array}
[c]{l}%
\mathbb{E}\left[  \left.  \left(  \int_{s}^{T}\left\vert q\left(  r\right)
\right\vert ^{2}\rho\left(  r\right)  \left\vert \hat{X}\left(  r\right)
\right\vert ^{2}dr\right)  ^{\frac{9}{8}}\right\vert \mathcal{F}_{s}%
^{t}\right]  \\
\leq\mathbb{E}\left[  \left.  \sup\limits_{r\in\lbrack t,T]}\left\vert \hat
{X}\left(  r\right)  \right\vert ^{\frac{9}{4}}\left(  \int_{s}^{T}\left\vert
q\left(  r\right)  \right\vert ^{2}\rho\left(  r\right)  dr\right)  ^{\frac
{9}{8}}\right\vert \mathcal{F}_{s}^{t}\right]  \\
\leq\left\{  \mathbb{E}\left[  \left.  \sup\limits_{r\in\lbrack t,T]}%
\left\vert \hat{X}\left(  r\right)  \right\vert ^{8}\right\vert \mathcal{F}%
_{s}^{t}\right]  \right\}  ^{\frac{9}{32}}\\
\text{ \ \ }\cdot\left\{  \mathbb{E}\left[  \left.  \left(  \int_{s}%
^{T}\left\vert q\left(  r\right)  \right\vert ^{2}\rho\left(  r\right)
dr\right)  ^{\frac{36}{23}}\right\vert \mathcal{F}_{s}^{t}\right]  \right\}
^{\frac{23}{32}}\\
=o\left(  \left\vert x-\bar{X}^{t,x;\bar{u}}\left(  s\right)  \right\vert
^{\frac{9}{4}}\right);
\end{array}
\]
\[%
\begin{array}
[c]{l}%
\mathbb{E}\left[  \left.  \left(  \int_{s}^{T}\left\vert q\left(  r\right)
\right\vert \left\vert \varphi\left(  r\right)  \right\vert ^{2}dr\right)
^{\frac{9}{8}}\right\vert \mathcal{F}_{s}^{t}\right]  \\
\leq C\mathbb{E}\left[  \left.  \sup\limits_{r\in\lbrack t,T]}\left\vert
\varphi\left(  r\right)  \right\vert ^{\frac{9}{4}}\left(  \int_{s}%
^{T}\left\vert q\left(  r\right)  \right\vert ^{2}dr\right)  ^{\frac{9}{16}%
}\right\vert \mathcal{F}_{s}^{t}\right]  \\
\leq\left\{  \mathbb{E}\left[  \left.  \sup\limits_{r\in\lbrack t,T]}%
\left\vert \varphi\left(  r\right)  \right\vert ^{\frac{5}{2}}\right\vert
\mathcal{F}_{s}^{t}\right]  \right\}  ^{\frac{9}{10}}\left\{  \mathbb{E}%
\left[  \left.  \left(  \int_{s}^{T}\left\vert q\left(  r\right)  \right\vert
^{2}dr\right)  ^{\frac{45}{8}}\right\vert \mathcal{F}_{s}^{t}\right]
\right\}  ^{\frac{1}{10}}\\
=o\left(  \left\vert x-\bar{X}^{t,x;\bar{u}}\left(  s\right)  \right\vert
^{\frac{9}{4}}\right);
\end{array}
\]%
\[%
\begin{array}
[c]{l}%
\mathbb{E}\left[  \left.  \left(  \int_{s}^{T}\left\vert P\left(  r\right)
\right\vert \left\vert \nu\left(  r\right)  \right\vert ^{2}dr\right)
^{\frac{9}{8}}\right\vert \mathcal{F}_{s}^{t}\right]  \\
\leq\mathbb{E}\left[  \left.  \sup\limits_{r\in\lbrack t,T]}\left\vert
P\left(  r\right)  \right\vert ^{\frac{9}{8}}\left(  \int_{s}^{T}\left\vert
\nu\left(  r\right)  \right\vert ^{2}dr\right)  ^{\frac{9}{8}}\right\vert
\mathcal{F}_{s}^{t}\right]  \\
\leq\left\{  \mathbb{E}\left[  \left.  \sup\limits_{r\in\lbrack t,T]}%
\left\vert P\left(  r\right)  \right\vert ^{\frac{45}{4}}\right\vert
\mathcal{F}_{s}^{t}\right]  \right\}  ^{\frac{1}{10}}\left\{  \mathbb{E}%
\left[  \left.  \left(  \int_{s}^{T}\left\vert \nu\left(  r\right)
\right\vert ^{2}dr\right)  ^{\frac{5}{4}}\right\vert \mathcal{F}_{s}%
^{t}\right]  \right\}  ^{\frac{9}{10}}\\
=o\left(  \left\vert x-\bar{X}^{t,x;\bar{u}}\left(  s\right)  \right\vert
^{\frac{9}{4}}\right),
\end{array}
\]
and the others are similar. The proof is completed.
\end{proof}

\begin{theorem}
Suppose the same assumptions as in Theorem \ref{th-visco-xx}. Then, for each
$s\in\lbrack t,T]$,
\[
\left\{
\begin{array}
[c]{c}%
\lbrack\mathcal{H}_{1}(s,\bar{X}^{t,x;\bar{u}}(s),\bar{Y}^{t,x;\bar{u}%
}(s),\bar{Z}^{t,x;\bar{u}}(s)),\infty)\subseteq D_{t+}^{1,+}W(s,X^{t,x;\bar
{u}}(s)),\\
D_{t+}^{1,-}W(s,X^{t,x;\bar{u}}(s))\subseteq(-\infty,\mathcal{H}_{1}(s,\bar
{X}^{t,x;\bar{u}}(s),\bar{Y}^{t,x;\bar{u}}(s),\bar{Z}^{t,x;\bar{u}}(s))],
\end{array}
\right.
\]
where%
\[
\begin{array}
[c]{l}%
\mathcal{H}_{1}(s,\bar{X}^{t,x;\bar{u}}(s),\bar{Y}^{t,x;\bar{u}}(s),\bar
{Z}^{t,x;\bar{u}}(s))\\
=-\mathcal{H}(s,\bar{X}^{t,x;\bar{u}}(s),\bar
{Y}^{t,x;\bar{u}}(s),\bar{Z}^{t,x;\bar{u}}(s),\bar{u}%
(s),p(t),q(t),P(t))+P(s)\sigma\left(  s\right)  ^{2}.
\end{array}
\]

\end{theorem}

\begin{proof}
The proof is the same as in Theorem \ref{th-visco-t} by using the estimates in
the proof of Theorem \ref{th-visco-xx}.
\end{proof}

\subsection{The local case}

In this case, the control domain is assumed to be a convex and compact set.
Note that in the above theorems, our control domain is only supposed to be a
nonempty and compact set. Then, for the local case we can still obtain the
relations in Theorem \ref{th-visco-x} under our Assumptions \ref{assum-1},
\ref{assum-2} and \ref{assum-3}. In this subsection, we study the MP by convex
variational method and its relationship with the DPP. For the convex variational
method, we suppose that $b$, $\sigma$ and $g$ are continuously differentiable
with respect to $u$, and we only need to consider the first-order variational
equation. So, every assumptions that guarantee the existence and uniqueness of
FBSDE (\ref{state-eq}) can be used in this case. Here we use the following
monotonicity conditions as in \cite{Wu98, Li-W}.

Define%
\[
\Pi(s,x,y,z,u)=\left(  -g,b,\sigma\right)  ^{\intercal}(s,x,y,z,u).
\]

\begin{assumption}
\label{assum-monotonic}There exist three nonnegative constants $\beta_{1}$,
$\beta_{2}$, $\beta_{3}$ such that $\beta_{1}+\beta_{2}>0$, $\beta_{2}%
+\beta_{3}>0$ and $\forall s\in\lbrack0,T]$, $\forall x$, $x^{\prime}$, $y$,
$y^{\prime}$, $z$, $z^{\prime}\in\mathbb{R}$, $\forall u\in U$,%
\[
\begin{array}
[c]{l}
\langle\Pi(s,x,y,z,u)-\Pi(s,x^{\prime},y^{\prime},z^{\prime},u),(x-x^{\prime
},y-y^{\prime},z-z^{\prime})^{T}\rangle \\
 \leq-\beta_{1}|x-x^{\prime}|^{2}%
-\beta_{2}(|y-y^{\prime}|^{2}+|z-z^{\prime}|^{2}),
\end{array}
\]%
\[
(\phi(x)-\phi(x^{\prime}))(x-x^{\prime})\geq\beta_{3}\left\vert x-x^{\prime
}\right\vert ^{2}.
\]
\end{assumption}

The adjoint equation in this case is the following linear FBSDE:%
\begin{equation}
\left\{
\begin{array}
[c]{rl}%
dh(s)= & \left[  g_{y}(s)h(s)+b_{y}(s)m(s)+\sigma_{y}(s)n(s)\right]
ds\\
&+\left[  g_{z}(s)h(s)+b_{z}(s)m(s)+\sigma_{z}(s)n(s)\right]  dB(s),\\
h(t)= & 1,\\
dm(s)= & -\left[  g_{x}(s)h(s)+b_{x}(s)m(s)+\sigma_{x}(s)n(s)\right]
ds+n(s)dB(s),\text{ }s\in\lbrack t,T],\\
m(T)= & \phi_{x}(\bar{x}(T))h(T).
\end{array}
\right.  \label{nesdda-1}%
\end{equation}
Define the following Hamiltonian function:%
\[
H^{\prime}(s,x,y,z,u,h,m,n)=mb(s,x,y,z,u)+n\sigma(s,x,y,z,u)+hg(s,x,y,z,u).
\]
Suppose Assumptions \ref{assum-1} (i) and \ref{assum-monotonic} hold. Let
$\bar{u}(\cdot)\in\mathcal{U}^{w}[t,T]$ be optimal for problem \eqref{obje-eq}
and $(h(\cdot),m(\cdot),n(\cdot))$ be the solution to FBSDE (\ref{nesdda-1}).
Then Wu \cite{Wu98} obtained the following MP:%
\begin{equation}
\begin{array}
[c]{l}%
\langle H_{u}^{\prime}(s,\bar{X}^{t,x;\bar{u}}(s),\bar{Y}^{t,x;\bar{u}%
}(s),\bar{Z}^{t,x;\bar{u}}(s),\bar{u}(s),h(s),m(s),n(s)),u-\bar{u}%
(s)\rangle\geq0,\ \\
\ \  \forall u\in U\ a.e.s\in\lbrack t,T],P\text{-}a.s..
 \end{array} \label{nesdda-2}
\end{equation}

\begin{theorem}
Suppose Assumptions \ref{assum-1} (i) and \ref{assum-monotonic} hold. Let
$\bar{u}(\cdot)$ be optimal for our problem \eqref{obje-eq} and $(h(\cdot
),m(\cdot),n(\cdot))$ be the solution to FBSDE (\ref{nesdda-1}). If $L_{3}$ is
small enough, then%
\[
D_{x}^{1,-}W(s,\bar{X}^{t,x;\bar{u}}(s))\subseteq\left\{  m(s)h^{-1}%
(s)\right\}  \subseteq D_{x}^{1,+}W(s,\bar{X}^{t,x;\bar{u}}(s)),\text{
}\forall s\in\lbrack t,T],\text{ }P\text{- }a.s..
\]
\end{theorem}

\begin{proof}
We use notations (\ref{nota-new-2}), (\ref{new-eqqw-11112}) and equations
(\ref{nota-new-1}), (\ref{eq-xy-hat}) in Step 1 in the proof of Theorem
\ref{th-visco-x}. By the estimate of FBSDE (see \cite{Li-W}), we obtain%
\[
\mathbb{E}\left[  \left.  \sup\limits_{r\in\lbrack s,T]}\left(  |\hat
{X}(r)|^{2}+|\hat{Y}(r)|^{2}\right)  +\int_{s}^{T}|\hat{Z}(r)|^{2}%
dr\right\vert \mathcal{F}_{s}^{t}\right]  \leq C\left\vert x^{\prime}-\bar
{X}^{t,x;\bar{u}}(s)\right\vert ^{2},P\text{-}a.s..
\]
Applying It\^{o}'s formula to $h(s)\hat{Y}(s)-m(s)\hat{X}(s)$, we get%
\[%
\begin{array}
[c]{l}%
h(s)\hat{Y}(s)-m(s)\hat{X}(s)\\
=\mathbb{E}\left[  \left.  h(T)\varepsilon_{4}(T)+\int_{s}^{T}(m(r)\varepsilon
_{1}(r)+n(r)\varepsilon_{2}(r)+h(r)\varepsilon_{3}(r))dr\right\vert
\mathcal{F}_{s}^{t}\right]  .
\end{array}
\]
Then, we want to prove $h(s)\hat{Y}(s)-m(s)\hat{X}(s)=o\left(  \left\vert
x^{\prime}-\bar{X}^{t,x;\bar{u}}(s)\right\vert \right)  $, and estimate the
terms in the right hand as follows%
\begin{align*}
\mathbb{E}\left[  \left.  |h(T)\varepsilon_{4}(T)|\right\vert \mathcal{F}%
_{s}^{t}\right]   &  \leq\left\{  \mathbb{E}\left[  \left.  |\hat{X}%
(T)|^{2}\right\vert \mathcal{F}_{s}^{t}\right]  \right\}  ^{1/2}\left\{
\mathbb{E}\left[  \left.  |h(T)(\tilde{\phi}_{x}(T)-\phi
_{x}(T))|^{2}\right\vert \mathcal{F}_{s}^{t}\right]  \right\}  ^{1/2}\\
&  =o\left(  \left\vert x^{\prime}-\bar{X}^{t,x;\bar{u}}(s)\right\vert
\right)  ;
\end{align*}%
\begin{align*}
&  \mathbb{E}\left[  \left.  \int_{s}^{T}|n(r)\left(  \tilde{\sigma}%
_{z}(r)-\sigma_{z}(r)\right)  \hat{Z}(r)|dr\right\vert
\mathcal{F}_{s}^{t}\right] \\
&  \leq\left\{  \mathbb{E}\left[  \left.  \int_{s}^{T}|n(r)\left(
\tilde{\sigma}_{z}(r)-\sigma_{z}(r)\right)  |^{2}dr\right\vert
\mathcal{F}_{s}^{t}\right]  \right\}  ^{1/2}\left\{  \mathbb{E}\left[  \left.
\int_{s}^{T}|\hat{Z}(r)|^{2}dr\right\vert \mathcal{F}_{s}^{t}\right]
\right\}  ^{1/2}\\
&  =o\left(  \left\vert x^{\prime}-\bar{X}^{t,x;\bar{u}}(s)\right\vert
\right)  .
\end{align*}
The estimates for the other terms are similar. Similar to Step 5 in the proof
of Theorem \ref{th-visco-x}, we can find a subset $\Omega_{0}\subseteq\Omega$
with $P(\Omega_{0})=1$ such that for any $\omega_{0}\in\Omega_{0}$,
\[
h(s,\omega_{0})\hat{Y}(s,\omega_{0})-m(s,\omega_{0})\hat{X}(s,\omega
_{0})=o\left(  \left\vert x^{\prime}-\bar{X}^{t,x;\bar{u}}(s,\omega
_{0})\right\vert \right)  \text{ for all }s\in\lbrack t,T].
\]
By the DPP in \cite{Li-W, Hu-JX-DPP}, we obtain
\[%
\begin{array}
[c]{ll}%
&W(s,x^{\prime})-W(s,\bar{X}^{t,x;\bar{u}}(s))\\
 & \leq Y^{s,x^{\prime};\bar{u}%
}(s)-\bar{Y}^{t,x;\bar{u}}(s)\\
& =\hat{Y}(s).\\
%& =m(s)h(s)^{-1}\left(  X^{s,x^{\prime};\bar{u}}(s)-\bar{X}^{t,x;\bar{u}%
%}(s)\right)  +o(|x^{\prime}-\bar{X}^{t,x;\bar{u}}(s)|).
\end{array}
\]
When $L_3$ small enough, by Theorem 5.4 in \cite{Hu-JX}, we can obtain $h(s)>0$ and $p(s)=m(s)h(s)^{-1}$.
Thus  $$\hat{Y}(s)
 =m(s)h(s)^{-1}\left(  X^{s,x^{\prime};\bar{u}}(s)-\bar{X}^{t,x;\bar{u}
}(s)\right)  +o(|x^{\prime}-\bar{X}^{t,x;\bar{u}}(s)|).$$
Since $x^{\prime}$ is arbitrary, from the definition of super-jet, we get%
\[
m(s)h(s)^{-1}\in D_{x}^{1,+}W(s,\bar{X}^{t,x;\bar{u}}(s)).
\]
Now we prove
\[
D_{x}^{1,-}W(s,\bar{X}^{t,x;\bar{u}}(s))\subseteq\left\{  m(s)h(s)^{-1}%
\right\}  .
\]
If $D_{x}^{1,-}W(s,\bar{X}^{t,x;\bar{u}}(s))$ is not empty, then taking any
$\xi\in D_{x}^{1,-}V(s,\bar{X}^{t,x;\bar{u}}(s))$, by definition of sub-jets,
we have%
\[%
\begin{array}
[c]{rl}%
0&\leq\underset{x^{\prime}\rightarrow\bar{X}^{t,x;\bar{u}}(s)}{\lim\inf
}\left\{  \frac{W(s,x^{\prime})-W(s,\bar{X}^{t,x;\bar{u}}(s))-\xi\left(
x^{\prime}-\bar{X}^{t,x;\bar{u}}(s)\right)  }{|x^{\prime}-\bar{X}^{t,x;\bar
{u}}(s)|}\right\} \\
&\leq\underset{x^{\prime}\rightarrow\bar{X}^{t,x;\bar{u}}(s)}{\lim\inf}\left\{
\frac{\left(  m(s)h(s)^{-1}-\xi\right)  \left(  x^{\prime}-\bar{X}%
^{t,x;\bar{u}}(s)\right)  }{|x^{\prime}-\bar{X}^{t,x;\bar{u}}(s)|}\right\}  .
\end{array}
\]
Thus we conclude that
\[
\xi=m(s)h(s)^{-1},\text{ }\forall s\in\lbrack t,T],\;\ P\text{-}a.s..
\]
The proof is completed.
\end{proof}

\begin{theorem}
\label{th-convex}Suppose Assumptions \ref{assum-1} (i) and
\ref{assum-monotonic} hold. Let $\bar{u}(\cdot)$ be optimal for problem
\eqref{obje-eq} and $(h(\cdot),m(\cdot),n(\cdot))$ be the solution to FBSDE
(\ref{nesdda-1}). If $L_{3}$ is small enough and the value function
$W(\cdot,\cdot)\in C^{1,2}([t,T]\times\mathbb{R})$, then%
\begin{equation}
\begin{array}[c]{ll}
\bar{Y}^{t,x;\bar{u}}(s)=W(s,\bar{X}^{t,x;\bar{u}}(s)),\\ \bar
{Z}^{t,x;\bar{u}}(s)=V(s,\bar{X}^{t,x;\bar{u}}(s),W(s,\bar{X}^{t,x;\bar{u}}(s)),W_x(s,\bar{X}^{t,x;\bar{u}}(s)),\bar{u}(s)),\text{ }%
s\in\lbrack t,T]
\end{array}
\label{eqdsaadf-1}%
\end{equation}
and for any $s\in\lbrack t,T]$,
\begin{equation}%
\begin{array}
[c]{l}%
-W_{s}(s,\bar{X}^{t,x;\bar{u}}(s))\\
  =G(s,\bar{X}^{t,x;\bar{u}}(s),W(s,\bar
{X}^{t,x;\bar{u}}(s)),W_{x}(s,\bar{X}^{t,x;\bar{u}}(s)),W_{xx}(s,\bar
{X}^{t,x;\bar{u}}(s)),\bar{u}(s))\\
 =\min\limits_{u\in U}G\left(  s,\bar{X}^{t,x;\bar{u}}(s),W(s,\bar
{X}^{t,x;\bar{u}}(s)),W_{x}(s,\bar{X}^{t,x;\bar{u}}(s)),W_{xx}(s,\bar
{X}^{t,x;\bar{u}}(s)),u),u\right).
\end{array}
\label{eqdsaadf-2}%
\end{equation}
Moreover, if $W(\cdot,\cdot)\in C^{1,3}([t,T]\times\mathbb{R})$ and
$W_{sx}(\cdot,\cdot)$ is continuous, then, for $s\in\lbrack t,T]$,%
\begin{equation}%
\begin{array}
[c]{rl}%
m(s)= & W_{x}(s,\bar{X}^{t,x;\bar{u}}(s))h(s),\\
n(s)= & \left(  1-W_{x}(s,\bar{X}^{t,x;\bar{u}}(s))\sigma_{z}\left(  s\right)
\right)  ^{-1}b_{z}(s)(W_{x}(s,\bar{X}^{t,x;\bar{u}}(s)))^{2}\\
& +g_{z}(s)W_{x}(s,\bar{X}^{t,x;\bar{u}})+W_{xx}(s,\bar{X}^{t,x;\bar{u}%
})\sigma(s)h(s),
\end{array}
\label{esddadf-1}%
\end{equation}
and  $\forall u\in U\ a.e.\;s\in\lbrack t,T],\;P\text{-}a.s.$
\begin{equation}
\langle H_{u}^{\prime}(s,\bar{X}^{t,x;\bar{u}}(s),\bar{Y}^{t,x;\bar{u}%
}(s),\bar{Z}^{t,x;\bar{u}}(s),\bar{u}(s),h(s),m(s),n(s)),u-\bar{u}%
(s)\rangle\geq0. \label{MP-convex}%
\end{equation}
\end{theorem}

\begin{proof}
The proof for (\ref{eqdsaadf-1}) and (\ref{eqdsaadf-2}) is the same as in
Theorem \ref{new-th-11}. Applying It\^{o}'s formula to $W_{x}(s,\bar
{X}^{t,x;\bar{u}}(s))h(s)$, one can check that $(h(\cdot),m(\cdot),n(\cdot))$
with $(m(\cdot),n(\cdot))$ given in (\ref{esddadf-1}) solves FBSDE
(\ref{nesdda-1}). By (\ref{eqdsaadf-2}), we have
\[%
\begin{array}
[c]{l}%
G(s,\bar{X}^{t,x;\bar{u}}(s),W(s,\bar{X}^{t,x;\bar{u}}(s)),W_{x}(s,\bar
{X}^{t,x;\bar{u}}(s)),W_{xx}(s,\bar{X}^{t,x;\bar{u}}(s)),\bar{u}(s))\\
\leq G\left(  s,\bar{X}^{t,x;\bar{u}}(s),W(s,\bar{X}^{t,x;\bar{u}}%
(s)),W_{x}(s,\bar{X}^{t,x;\bar{u}}(s)),W_{xx}(s,\bar{X}^{t,x;\bar{u}%
}(s)),u),u\right) \\
 \text{ }\forall u\in U\ a.e.,\ a.s..
\end{array}
\]
Thus we obtain
\begin{small}
\[%
\begin{array}
[c]{l}%
\left\langle \left.  \frac{\partial}{\partial u}G\left(  s,\bar{X}%
^{t,x;\bar{u}}(s),W(s,\bar{X}^{t,x;\bar{u}}(s)),W_{x}(s,\bar{X}^{t,x;\bar{u}%
}(s)),W_{xx}(s,\bar{X}^{t,x;\bar{u}}(s)),u\right)  \right\vert _{u=\bar
{u}(s)},\right.  \\
\text{ } \left.  u-\bar{u}\left(  s\right)  \right\rangle \geq0,\forall u\in
U\ a.e.,\ a.s.,
\end{array}
\]
\end{small}
which implies
\begin{equation}%
\begin{array}
[c]{l}%
\left\langle \left\{  W_{x}(s,\bar{X}^{t,x;\bar{u}}(s))\left[  b_{z}\left(
s\right)  V_{u}(s)+b_{u}(s)\right]
\right.  \right. \\
+W_{xx}(s,\bar{X}^{t,x;\bar{u}}(s))\sigma(s)\left[  \sigma_{z}\left(
s\right)  V_{u}(s)+\sigma_{u}(s)\right]
\\
\left.  \left.  +g_{z}\left(  s\right)  V_{u}(s)\right\}  ,u-u\left(  s\right)  \right\rangle
\geq0,\forall u\in U\ a.e.,\ a.s.,
\end{array}
\label{asds-121}%
\end{equation}
where $$ V_u(s)=\left.\frac{\partial V}{\partial u}(s,\bar{X}^{t,x;\bar{u}}(s),W(s,\bar{X}^{t,x;\bar{u}}(s)),W_x(s,\bar{X}^{t,x;\bar{u}}(s)),u)\right\vert_{u=\bar{u}(s)}.$$
%Noting that
%\[
%\begin{array}
%[c]{l}%
%V(s,\bar{X}^{t,x;\bar{u}}(s),W(s,\bar{X}^{t,x;\bar{u}}(s)),W_x(s,\bar{X}^{t,x;\bar{u}}(s)),u)\\
%=W_{x}(s,\bar{X}^{t,x;\bar{u}}(s))\sigma
%(s,\bar{X}^{t,x;\bar{u}}(s),W(s,\bar{X}^{t,x;\bar{u}}(s)),V(s,\bar{X}^{t,x;\bar{u}}(s),W(s,\bar{X}^{t,x;\bar{u}}(s)),W_x(s,\bar{X}^{t,x;\bar{u}}(s)),u),u),
%\end{array}
%\]
%then,
By implicit function theorem, we deduce that%
\begin{equation}
\begin{array}
[c]{l}
\left.\frac{\partial V}{\partial u}(s,\bar{X}^{t,x;\bar{u}}(s),W(s,\bar{X}^{t,x;\bar{u}}(s)),W_x(s,\bar{X}^{t,x;\bar{u}}(s)),u)\right\vert_{u=\bar{u}(s)}\\
=\left(  1-W_{x}(s,\bar
{X}^{t,x;\bar{u}}(s))\sigma_{z}\left(  s\right)  \right)  ^{-1}W_{x}(s,\bar
{X}^{t,x;\bar{u}}(s))\sigma_{u}(s).
\end{array}\label{asds-122}%
\end{equation}
Combing (\ref{esddadf-1}), (\ref{asds-121}) and (\ref{asds-122}), we obtain
the desired results (\ref{MP-convex}).
\end{proof}

\begin{remark}
From Theorems \ref{new-th-11} and \ref{th-convex}, we can obtain the following
relationship between $(p\left(  \cdot\right)  ,q\left(  \cdot\right)  )$ and
$\left(  h\left(  \cdot\right)  ,m(\cdot),n\left(  \cdot\right)  \right)  $:%
\begin{align*}
m(s)  &  =p(s)h(s);\\
n(s)  &  =\left(  1-p(s)\sigma_{z}(s)\right)  ^{-1}[b_{z}(s)p(s)^{2}%
+p(s)g_{z}(s)+q(s)]h(s).
\end{align*}

\end{remark}

\end{document}